\documentclass[10pt]{article}
\usepackage[english,activeacute]{babel}
\usepackage[T1]{fontenc}
\usepackage{amsmath,amssymb,amsthm,graphicx}
\newcommand{\X}{\mathcal{X}}
\newcommand{\R}{\mathbb{R}}
\newcommand{\C}{\mathbb{C}}
\newcommand{\N}{\mathbb{N}}

\newcommand{\ptl}{{\partial}}
\providecommand{\abs}[1]{\lvert#1 \rvert}
\providecommand{\pare}[1]{$($#1$)$}
\newcommand{\wh}{\widehat}
\newcommand{\E}{\mathcal E}

\newcommand{\M}{\mathcal M}
\renewcommand{\H}{\dot{H}}

\newcommand{\lp}{\textup{(}}
\renewcommand{\mod}{\textup{ mod }}
\newcommand{\rp}{\textup{)} }
\newcommand{\loc}{\operatorname{loc}}
\providecommand{\norm}[1]{\lVert#1 \rVert}

\renewcommand{\Re}{\operatorname{Re}}
\renewcommand{\Im}{\operatorname{Im}}
\newcommand{\supp}{\operatorname{supp}}
\newcommand{\ve}{\varepsilon}
\DeclareMathOperator*{\essinf}{ess\,inf}
\providecommand{\llave}[1]{\langle\negthickspace\langle #1 \rangle\negthickspace\rangle}           
\newtheorem{teo}{Theorem}[section]
\newtheorem{definition}[teo]{Definition}
\newtheorem{prop}[teo]{Proposition}
\newtheorem{lema}[teo]{Lemma}

\newtheorem{cor}[teo]{Corollary}
\theoremstyle{definition}
\newtheorem{remark}[teo]{Remark}
\newcommand{\bqq}{\begin{equation*}}
\newcommand{\eqq}{\end{equation*}}
\newcommand{\bq}{\begin{equation}}
\newcommand{\eq}{\end{equation}}
\newcommand{\grad}{\nabla}
\makeatletter
\newenvironment{ecu0}{\begin{equation}\begin{aligned}}{\end{aligned}\end{equation}\@ignoretrue}
\newenvironment{ecu}{\begin{equation}\left\lbrace\begin{aligned}}{\end{aligned}\right.\end{equation}\@ignoretrue}
\newenvironment{ecut}[1]{\begin{equation}\tag{#1}\left\lbrace\begin{aligned}}{\end{aligned}\right.\end{equation}\@ignoretrue} 
\makeatother
\long\def\symbolfootnote[#1]#2{\begingroup%
\def\thefootnote{\fnsymbol{footnote}}\footnote[#1]{#2}\endgroup} 
\providecommand{\LL}[2]{\norm{#1}_{L^{#2}}}
\newcommand{\HH}[1]{\norm{#1}_{H^1}}

\author{Andr\'e de Laire\\
{UPMC Univ Paris 06, UMR 7598}\\
{Laboratoire Jacques-Louis Lions, F-75005, Paris, France}\\
{delaire@ann.jussieu.fr}
}
\title{Global well-posedness  for a nonlocal Gross-Pitaevskii equation with non-zero condition at infinity}
\date{}
\begin{document}
\maketitle
\begin{abstract}
\noindent\emph{We study the Gross-Pitaevskii equation involving
a nonlocal interaction potential. Our aim is to give sufficient
conditions that cover  a variety of nonlocal interactions
such that the associated Cauchy problem is
globally well-posed with non-zero boundary condition at infinity,
in any dimension. We focus on even potentials
that are positive definite or positive tempered distributions.}\\

\noindent{{\bf Keywords} Nonlocal Schr\"odinger equation; Gross-Pitaevskii equation; Global well-posedness;
Initial value problem.}\\

\noindent{\bf Mathematics Subject Classification}  35Q55; 35A05; 37K05; 35Q40; 81Q99.\\

\end{abstract}

\section{Introduction}
\subsection{The problem} 
In order to describe the kinetic of a weakly interacting Bose gas
of bosons of mass $m$, Gross \cite{gross} and 
Pitaevskii \cite{pitaevskii} derived in the Hartree approximation,
{that the wavefunction $\Psi$ governing the condensate satisfies}
\begin{equation}
i\hbar\partial_t\Psi(x,t)=-\frac{\hbar^2}{2m}\Delta\Psi(x,t)+\Psi(x,t) \int_{\R^N} \abs{\Psi(y,t)}^2V(x-y)\,dy, 
\text{ on $\R^N\times \R,$}
\label{GP-full}
\end{equation}
where $N$ is the space dimension and $V$ describes the interaction between bosons.
In the most typical first approximation, $V$ is considered
as a Dirac delta function, which leads to the standard local Gross-Pitaevskii
equation. This local model with non-vanishing condition at infinity has been intensively used,
due to its application in various  areas of physics, such as
superfluidity, nonlinear optics and Bose-Einstein condensation \cite{JPR1,JPR2,kivshar,coste}.
It seems then natural to analyze the equation \eqref{GP-full} for more
general interactions. Indeed, in the study of superfluidity, supersolids
and Bose-Einstein condensation,
 different types of nonlocal potentials have been proposed \cite{berloff0,deconinck,kraenkel,rica0,rica,aftalion,yi,cuevas,remi}.

To obtain a dimensionless equation, we take the average energy level per unit mass $\E_0$ of a boson, and we set
$$\psi(x,t)=\exp\left(\frac{i m \E_0 t}{\hbar}\right)\Psi(x,t).$$
Then \eqref{GP-full} turns into
\begin{equation}
i\hbar \partial_t\psi(x,t)=-\frac{\hbar^2}{2m}\Delta\psi(x,t)-m\E_0 \psi(x,t)+\psi(x,t) \int_{\R^N} \abs{\psi(y,t)}^2V(x-y)\,dy.
\label{GP-intro0}
\end{equation}
Defining the rescaling
$$u(x,t)=\frac{1}{\lambda \sqrt{m\E_0}}\left(\frac{\hbar}{\sqrt{2m^2\E_0}}\right)^{\frac N2}\psi\left(\frac{\hbar x}{\sqrt{2m^2\E_0}},\frac{\hbar t}{m\E_0}\right),$$
from \eqref{GP-intro0} we deduce that
\begin{equation*}
i\partial_t u(x,t)+\Delta u(x,t)+u(x,t)\left(1-\lambda^2\int_{\R^N}\abs{u(y,t)}^2 \mathcal V(x-y)\,dy\right)=0,
\end{equation*}
with $$\mathcal V(x)=V\left(\frac{\hbar x}{\sqrt{2m^2\E_0}}\right).$$
If we assume that the convolution
between $\mathcal V$ and a constant is well-defined  and equal to a positive constant,
choosing $\lambda^2=(\mathcal V*1)^{-1},$ equation \eqref{GP-intro0} is equivalent to
\begin{equation}\label{GP-intro}
i \ptl _t u+\Delta u+\lambda^2 u(\mathcal V*(1-\abs{u} ^2))=0 \textrm{ on } \R^N\times \R.\\
\end{equation}
More generally, we consider the Cauchy problem for the nonlocal Gross-Pitaevskii equation
with non-zero initial condition at infinity in the form
\begin{equation} \label{NGP}\tag{NGP}
\left\lbrace\begin{aligned}
i \ptl _t u+\Delta u+u(W*(1-\abs{u} ^2))&=0 \textrm{ on } \R^N\times \R,\\
u(0)&=u_0,
\end{aligned}\right.
\end{equation}
where
\begin{equation}\label{BC-1}
 \abs{u_0(x)} \to 1, \quad\textrm{as} \quad \abs{x}\to\infty.
\end{equation}

If $W$ is a real-valued even distribution, \eqref{NGP} is
a Hamiltonian equation whose energy given by
\begin{equation*}
E(u(t))=\frac12 \int_{\R^N}\abs{\nabla u(t)}^2\,dx +\frac 14 \int_{\R^N}(W*(1-\abs{u(t)}^2))(1-\abs{u(t)}^2)\,dx
\end{equation*}
is formally  conserved.

In the case that $W$ is the Dirac delta function, \eqref{NGP} corresponds to
the local Gross-Pitaevskii equation
and the Cauchy problem in this instance has been studied by B\'ethuel and Saut \cite{bethuel2},
G\'erard  \cite{gerard}, Gallo  \cite{gallo},
 among others.
As mentioned before, in a more
general framework
the interaction kernel $W$ could be nonlocal. For example,  Shchesnovich and Kraenkel in \cite{kraenkel} consider
for $\ve>0$,
\begin{equation}W_\ve({x})=\begin{cases}
\dfrac{1}{2\pi \ve^2 \abs{x}}K_0\left(\dfrac{\abs{x}}{\ve}\right), \quad  N=2,\\
\dfrac{1}{4\pi \ve^2 \abs{x}}\exp{\left(-\dfrac{\abs{x}}{\ve}\right)}, \quad  N=3,
\end{cases}
\label{ex:W}
\end{equation}
where $K_0$ is the modified Bessel function of second kind (also called Macdonald function).
In this way $W_\ve$ might be considered as an approximation of the Dirac delta function, since $W_\ve\to \delta$,
as $\ve\to 0,$ in a distributional sense.
Others interesting nonlocal interactions are
the soft core potential \begin{equation}
W(x)=\begin{cases}
1, & \ \text{ if }\abs{x}<a,\\
0, & \ \text{otherwise},
\end{cases}
\label{ex-rica}
\end{equation}
with $a>0$, which is used in \cite{rica,aftalion} to the study of supersolids{,} and also
\bq
W=\alpha_1\delta+\alpha_2 K,\qquad \alpha_1,\alpha_2\in \R, \label{dipolar2}
\eq
 where $K$ is the singular kernel
\bq \label{dipolar} K(x)=\dfrac{x_1^2+x_2^2-2x_3^2}{\abs{x}^5}, \quad x\in\R^3\backslash\{0\}.\eq
The potential \eqref{dipolar2}-\eqref{dipolar} models dipolar forces in a quantum gas (see \cite{remi}, \cite{yi}).

\subsection{Main results}
In order to include interactions such as \eqref{dipolar2}-\eqref{dipolar}, it is appropriate to work in the
space $\M_{p,q}(\R^N)$, that is the set of tempered distributions $W$ such that the linear operator
$f\mapsto W*f$ is bounded from $L^p(\R^N)$ to  $L^q(\R^N)$. We denote by $\norm{W}_{p,q}$ its norm. We will suppose that
there exist
$$p_1,p_2,p_3,p_4,q_1,q_2,q_3,q_4,s_1,s_2\in [1,\infty),$$
with
$$ \frac{N}{N-2}>p_4, \ \  \frac{2N}{N-2}>p_2,p_3,s_1,s_2\geq 2,\ \
2\geq q_1 >\frac{2N}{N+2},\ \ q_3,q_4>\frac{N}{2}  \quad \text{ if } N\geq 3$$
 and
$$p_2,p_3,s_1,s_2\geq 2,\ \ 2\geq q_1>1 \quad \text{ if } 2\geq N\geq 1,$$
such that
\begin{ecut}{$\mathcal W_N$} \label{Wn}
&W\in  \M_{2,2}(\R^N)\cap \displaystyle\bigcap\limits_{i=1}^4 \M_{p_i,q_i}(\R^N),\\
&\dfrac{1}{p_3}+\dfrac{1}{q_2}=\dfrac{1}{q_1},\quad \dfrac{1}{p_1}-\dfrac{1}{p_3}=\dfrac{1}{s_1}, \quad \dfrac{1}{q_1}-\dfrac{1}{q_3}=\dfrac{1}{s_2}
 \quad \text{ if } N\geq 3. 
\end{ecut}
We recall that if $p>q$, then $\M_{p,q}=\{0\}$. Therefore {if we suppose that} $W$ is not zero,
 the numbers above have to satisfy $q_2, q_3\geq 2$. In addition, the existence of $s_1$, $s_2$ and the relations in \eqref{Wn} imply that  
$$\frac{N}{N-2}>p_1, \ \ q_2>\frac{N}{2}, \ \ \frac1{p_1}-\frac{1}{p_3}\in\left(\frac{N-2}{2N},\frac12 \right], 
\ \ 
\frac1{q_1}-\frac{1}{q_3}\in\left(\frac{N-2}{2N},\frac12 \right]
\quad \text{ if } N\geq 3.$$
Figure~\ref{fig} schematically shows the location of these numbers in the unit square.
\begin{figure}[ht!]
\begin{center}
\includegraphics{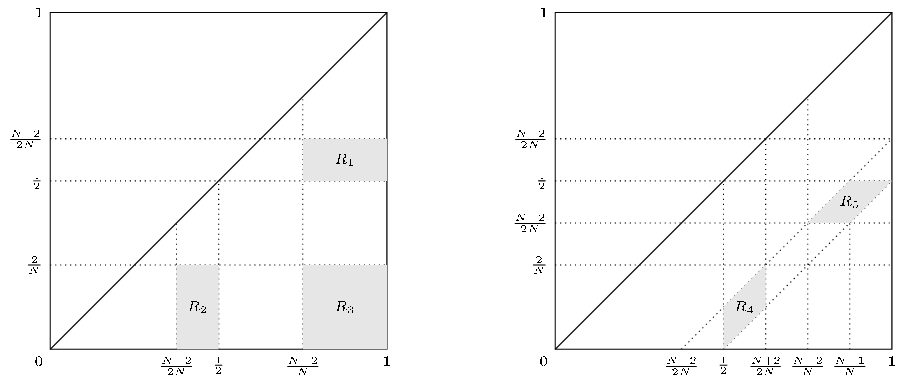} 
\end{center}
 \caption{\em{For $N>4$, the picture on the left represents the $(1/p,1/q)$-plane, in the sense that
  $(1/p_1,1/q_1)\in R_1$, $(1/p_2,1/q_2),(1/p_3,1/q_3)\in R_2$, $(1/p_4,1/q_4)\in R_3$.
 In the picture on the right, the shaded areas symbolize that $(1/q_1,1/q_3)\in R_4$ and $(1/p_1,1/p_3)\in R_5$, for
 $N>6$.}}
 \label{fig}

\end{figure}


To check the hypothesis \eqref{Wn} it is convenient to use some properties of the spaces
$\M_{p,q}(\R^N)$. For instance,  for any $1<p\leq q<\infty$, $\M_{p,q}(\R^N)=\M_{q',p'}(\R^N)$
and  for any $1\leq p\leq 2$, $\M_{1,1}(\R^N)\subseteq \M_{p,p}(\R^N) \subseteq \M_{2,2}(\R^N)$ (\cite{grafakos}).
In Proposition~\ref{regularidad} we give more explicit conditions to ensure \eqref{Wn}.

As remarked before, the energy is formally conserved
if $W$ is a real-valued even distribution. We recall
that a real-valued distribution is said to be even if
\bqq
\langle W,\phi\rangle=\langle W,\widetilde \phi\rangle, \quad \forall  \phi \in C_0^\infty(\R^N;\R),
\eqq
where $\widetilde \phi(x)=\phi(-x)$.
However, the conservation of energy is not sufficient  to study the long time behavior of the Cauchy problem,
 because  the potential energy is not
necessarily nonnegative and the nonlocal nature of the problem
prevents us to obtain pointwise bounds. We are able to control this term
assuming further that $W$ is a {\em positive distribution} or supposing that it is a
{\em positive definite distribution}. More precisely, we say that $W$ is a positive
distribution if
$$\langle  W,\phi\rangle\geq 0, \quad  \forall \phi \geq 0,\ \phi \in C_0^\infty(\R^N;\R),$$
and that it is a positive definite distribution if
\bq\label{def-pos} \langle W,\phi*\widetilde{\phi} \rangle \geq 0, \quad \phi \in C_0^\infty(\R^N;\R).\eq
These type of distributions frequently arise in the physical models (see Subsection~\ref{examples}).
In particular, the real-valued even positive definite distributions include
a large variety of models where the interaction between particles is symmetric.
In Section~\ref{section-pos} we state further properties
of this kind of potentials.

As Gallo in \cite{gallo}, we consider the initial data $u_0$ for the problem
\eqref{NGP} belonging to the space
${\phi}+H^1(\R^N)$, with $\phi$ a function of finite energy. More precisely,
from now on we assume that $\phi$ is a complex-valued function that satisfies
\begin{equation}
\phi\in W^{1,\infty}(\R^N), \ \grad \phi\in H^2(\R^N)\cap C(B^c), \ \abs{\phi}^2-1\in L^2(\R^N),
\label{phi}
\end{equation}
{where $B^c$ denotes the complement of} some ball $B\subseteq\R^N$, so that in particular $\phi$
satisfies \eqref{BC-1}.
\begin{remark}\label{rem-1}
We do not suppose that $\phi$ has a limit at infinity. In dimensions $N=1,2$
a function satisfying \eqref{phi} could have complicated oscillations, such as
(see \cite{gerard,gerard3})
\bqq
\phi(x)=\exp(i(\ln(2+\abs{x}))^\frac14), \quad x\in \R^2.\eqq

We note that any function verifying \eqref{phi} belongs to the Homogeneous Sobolev space
\bqq
 \H^1(\R^N)=\{\psi\in L^2_{\loc}(\R^N) : \grad \psi \in L^2(\R^N)\}.
\eqq
In particular, if $N\geq 3$ there exists $z_0\in \C$ with $\abs{z_0}=1$
such that $\phi-z_0\in L^{\frac{2N}{N-2}}(\R^N)$ (see e.g. Theorem 4.5.9 in \cite{hormander}).
Choosing  $\alpha \in \R$
such that $z_0=e^{i\alpha}$ and since the equation \eqref{NGP} is invariant by a  phase change, one can assume that
$\phi-1\in L^{\frac{2N}{N-2}}(\R^N)$, but we do not use explicitly  this decay in order to handle
at the same time the two-dimensional case.
\end{remark}

Our main result concerning the global well-posedness for the Cauchy problem is the following.
\begin{teo}\label{global}
Let $W$ be a real-valued even distribution satisfying \eqref{Wn}.
\begin{itemize}
\item[$(i)$] Assume that one of the following is verified
\begin{itemize}
\item[$(a)$]  $N\geq 2$\textrm{ and }$W$\textrm{ is a positive definite distribution}.  
\item[$(b)$] $N\geq 1$, $W\in\M_{1,1}(\R^N)$\textrm{ and }$W$\textrm{ is a positive distribution}.
\end{itemize}
Then  the Cauchy problem \eqref{NGP}
is globally well-posed in $\phi+H^1(\R^N)$. More precisely,
for every $w_0\in H^1(\R^N)$  there
exists a unique  $w \in C(\R,H^1(\R^N))$, for which $\phi+w$
solves \eqref{NGP} with the initial condition $u_0=\phi+w_0$
and for any bounded {closed} interval $I\subset\R$, the flow map $w_0\in H^1(\R^N) \mapsto w \in C(I,H^1(\R^N))$
is continuous.
Furthermore, $w \in C^1(\R,H^{-1}(\R^N))$ and the energy is conserved
\bq\label{E-cte}
E_0:=E(\phi+w_0)=E(\phi+w(t)), \ \forall t\in \R.
\eq
\item[$(ii)$] Assume that there exists $\sigma>0$ such that
\bq\label{hat-W} \essinf \widehat{W}\geq\sigma.\eq
Then  \eqref{NGP} is globally well-posed in $\phi+H^1(\R^N)$, for all $N\geq 1$ and \eqref{E-cte} holds. Moreover,
if $u$ is the solution associated to the initial data $u_0\in \phi+H^1(\R^N)$, we have
the growth estimate
\bq\label{linear} \LL{u(t)-\phi}{2}\leq C \abs{t}+\LL{u_0-\phi}{2},\eq
for any $t\in \R$, where $C$ is a positive constant that depends only on $E_0$,
$W,$ $\phi$ and $\sigma$.
\end{itemize}

\end{teo}
We make now some remarks about Theorem~\ref{global}.
\begin{itemize}
\item The condition \eqref{Wn} implies that $W\in \M_{2,2}(\R^N)$, so that $\wh W\in L^\infty(\R^N)$ and therefore
the condition \eqref{hat-W} makes sense.
\item In contrast with \eqref{linear}, as we prove in Section~\ref{sec-global}, the growth estimate
for the solution given by Theorem~\ref{global}-(i) is only exponential
\bqq \LL{u(t)-\phi}{2}\leq C_1 e^{C_2\abs{t}}(1+\LL{u_0-\phi}{2}) ,\qquad t\in \R,\eqq
for some constants $C_1,C_2$ only depending on $E_0$,
$W$ and $\phi$.
\item Accordingly to Remark~\ref{rem-1} and the Sobolev embedding theorem, after
a  {phase change} independent of $t$, the solution $u$ of \eqref{NGP} given by Theorem~\ref{global}
 also satisfies that $u-1\in L^{\frac{2N}{N-2}}(\R^N)$ if $N\geq 3$.

\item
In dimensions $1\leq N\leq 3$
we can choose $(p_4,q_4)=(2,2)$ in \eqref{Wn}. Consequently, the condition that  $W\in \M_{p_4,q_4}(\R^N)$
is  nontrivial {only} when $N\geq 4$.
\end{itemize}

At first sight, it is not obvious to check the hypotheses on $W$.
The purpose of the next result is to give sufficient conditions to ensure \eqref{Wn}.
\begin{prop}~\label{regularidad}
\begin{itemize}
\item[$(i)$] Let $1\leq N\leq 3$. If $W\in \M_{2,2}(\R^N)\cap \M_{3,3}(\R^N)$,
then $W$ fulfils \eqref{Wn}. Furthermore, if $W$ verifies \eqref{Wn} with $p_i=q_i$,
$1\leq i\leq {3}$,
then $W\in \M_{2,2}(\R^N)\cap \M_{3,3}(\R^N)$.

\item[$(ii)$] Let $N\geq 4$. Assume that  $W\in \M_{r,r}(\R^N)$ for {every} $1<r<\infty$.
Also suppose that there exists $\bar r>\frac{N}4$ such that
$W\in \M_{p,q}(\R^N)$, for {every} $1-\frac1{\bar r}<\frac1p<1$
with $\frac{1}{q}=\frac{1}{p}+\frac{1}{\bar{r}}-1$. Then $W$ satisfies \eqref{Wn}.
\end{itemize}
\end{prop}

We conclude from Proposition \ref{regularidad} that the Dirac delta function verifies \eqref{Wn} in dimensions $1\leq N\leq 3$. Since
$\wh\delta=1$, Theorem~\ref{global}-(ii)
recovers the results of global existence for the local Gross-Pitaevskii equation in \cite{bethuel2,gerard,gallo}
and the growth estimate proved in \cite{banica}. In addition, if the potential converges to the
Dirac delta function, the correspondent solutions converge to the solution of the local problem as a consequence
of the following result.

\begin{prop}\label{convergencia}
Assume that $1\leq N\leq 3$. Let $(W_n)_{n\in \N}$ be a sequence of real-valued distributions in
$\M_{2,2}(\R^N)\cap \M_{3,3}(\R^N)$ 
such that $u_n$ is {the} global solution of \eqref{NGP} given by Theorem~\ref{global}, 
with $W_n$ instead of $W,$ for some initial data in $\phi+ H^1(\R^N)$, and
\bq\label{conver} \lim_{n\to\infty}W_n={W_\infty}, \quad \textup{ in  } \ \M_{2,2}(\R^N)\cap \M_{3,3}(\R^N),\eq
with $\norm{W_\infty}_{\M_{2,2}\cap\M_{3,3}}>0$
\textup{(}$\norm{\cdot}_{\M_{2,2}\cap\M_{3,3}}:=\max\{\norm{\cdot}_{\M_{2,2}},\norm{\cdot}_{\M_{3,3}}\}$\textup{)}.
 Then $u_n\to u$ in $C(I,H^1(\R^N))$,
for any bounded closed interval $I\subset\R$,
where $u$ is the solution of \eqref{NGP} with $W={W_\infty}$
and the same initial data.
\end{prop}

On the other hand, the Dirac delta function does not satisfy \eqref{Wn} if $N\geq 4$ and therefore
Theorem~\ref{global} cannot be applied. In fact, to our knowledge there is no proof for the global well-posedness
to the local Gross-Pitaevskii equation in dimension $N\geq 4$ with arbitrary initial condition.
For small initial data, Gustafson et al. \cite{nakanishi} proved global well-posedness in dimensions $N\geq 4$
as well as G\'erard \cite{gerard}  in the four-dimensional energy space.

As a consequence of Theorem \ref{global} and Proposition \ref{regularidad} we derive the next result for integrable kernels.
\begin{cor}\label{cor}
Let $W$ be a real-valued even function such {that} $W\in L^1(\R^N)$
if \mbox{$1\leq N\leq 3$} and $W\in L^1(\R^N)\cap L^{r}(\R^N)$, for some $r>\frac{N}{4}$,
if $N\geq 4$. Assume also that
$W$  is positive definite if  $N\geq 2$, or  that it is nonnegative.
Then  the Cauchy problem \eqref{NGP}
is globally well-posed in $\phi+H^1(\R^N)$.
\end{cor}

As Gallo remarks in \cite{gallo}, the well-posedness
in a space such as $\phi+H^1(\R^N)$ makes possible to handle
the problem with initial data in the energy space
$$\mathcal E(\R^N)=\{u\in H^1_{\textrm{loc}}(\R^N) : \grad u\in L^2(\R^N), \ 1-\abs{u}^2\in L^2(\R^N) \},$$
equipped with the distance
\bq\label{distance} {d}(u,v)=\norm{u-v}_{X^1+H^1}+\LL{\abs{u}^2-\abs{v}^2}{2}.
\eq
Here $X^1(\R^N)$ denotes the Zhidkov space
$$X^1(\R^N)=\{u\in L^\infty(\R^N) :  \grad u\in L^2(\R^N)\}. $$
We recall that $u\in C(\R,\mathcal E(\R^N))$ is called a \emph{mild solution}
of \eqref{NGP} if it satisfies the Duhamel formula
$${u(t)=e^{it\Delta} u_0+i\int_0^t e^{i(t-s)\Delta}(u(s)(W*(1-\abs{u(s)}^2))\,ds, \quad t\in \R.}$$
{We note that by Lemma~\ref{lema4} the integral in the r.h.s is actually  finite}
(see  \cite{gerard,gerard3} for further results about the action of Schr\"odinger semigroup
on $\E(\R^N)$). %
With the same arguments of \cite{gallo}, we may 
also handle the problem with initial data 
in the energy space. Moreover, in the case $1\leq N\leq 4$, 
{we prove that a solution in the energy space with initial condition} $u_0\in \E(\R^N)$, necessarily
belongs to $u_0+H^1(\R^N)$, which is a proper subset of $\E(\R^N)$. 
This also gives the uniqueness in the energy space for $1\leq N\leq 4$, as follows.
\begin{teo}\label{teo-gallo}
Let $W$ be as in Theorem~\ref{global}. Then for any $u_0\in \mathcal E(\R^N)$, there exists a unique
$w\in C(\R, H^1(\R^N))$ such that $u:=u_0+w$ solves \eqref{NGP}.
Furthermore, if $1\leq N\leq 4$ and $v\in C(\R,\mathcal E(\R^N))$ is a mild solution of \eqref{NGP}
with $v(0)=u_0$, then $v=u$.
\end{teo}

The next proposition shows that the hypotheses made on the potential $W$ also ensure the
 $H^2$-regularity of the solutions.
\begin{prop}\label{prop-H2}
Let $W$ be as in Theorem~\ref{global} and
 $u$ be the global solution of \eqref{NGP} for some
initial data $u_0\in \phi+H^2(\R^N)$. Then $u-\phi\in C(\R,H^2(\R^N))\cap C^1(\R,L^2(\R^N))$.
\end{prop}

Finally, we study the
conservation of  momentum and  mass for \eqref{NGP}.
As has been discussed in several works (see {\cite{bethuel0,bethuel-black,maris,bethuel3})
the classical concepts of momentum and mass, that is
\bqq
p(u)=\int_{\R^N}{\llave{i \grad u,u}} \,dx \quad\text{ and } \quad M(u)=\int_{\R^N}(1-\abs{u}^2)\,dx,\eqq
with ${\llave{z_1,z_2}=\Re(z_1 \overline{z}_2)}$,
are not well-defined for $u\in \phi+H^1(\R^N)$.
Thus it is necessary to give some generalized sense to
 these quantities. In Section~\ref{conservation} we will explain in detail a notion
of {\em generalized momentum}  and {\em generalized mass} such that we have
the next results on conservation laws.

\begin{teo}\label{teo-momentum}
Let $N\geq 1$ and $u_0\in \phi +H^1(\R^N)$. Then the generalized momentum is conserved by the flow
of the associated solution $u$ of \eqref{NGP} given by Theorem~\ref{global}.
\end{teo}

\begin{teo}\label{teo-masa}
Let $1\leq N\leq 4$. In addition to \eqref{phi}, assume that
$\grad \phi\in L^\frac{N}{N-1}(\R^N)$ if $N=3,4$.
Suppose that $u_0\in \phi+H^1(\R^N)$
has finite generalized mass.
Then the generalized mass of the associated solution of \eqref{NGP} given by Theorem~\ref{global}
is conserved by the flow.
\end{teo}
\subsection{Examples}\label{examples}
\begin{enumerate}
\item[(i)]
Given the spherically symmetric
interaction of bosons, it is usual to suppose that $W$ is radial, that is
$W(x-y)=R(\abs{x-y}),$
with $R:[0,\infty)\to \R$. Using the fact that the Fourier transform of
a radial function is also radial, we may write
$\widehat W(\xi)=\rho(\abs{\xi}),$
for some function $\rho:[0,\infty)\to \R$. Noticing that
$\wh \delta=1$, a next order of approximation would be to consider (see e.g.~\cite{kraenkel})
$$\rho(r)=\frac{1}{1+\ve^2 r^2}, \qquad \ve>0.$$
Then the Fourier inversion theorem implies that $W$ is given by \eqref{ex:W} for $N=2,3$.
By Proposition~\ref{lema-def}, $\eqref{ex:W}$ is indeed a positive definite function, since $\rho$ is nonnegative.
For this potential we also have that $K_0(x)\approx \ln\left(\frac2x\right)$ as $x\to 0$,
and $K_0(x)\approx \sqrt{\frac{\pi}{2x}}\exp(-x)$ as $x\to \infty$ (see e.g.~\cite{lebedev}, p.~136), hence  $W\in L^1(\R^N)$
for $N=2,3$. Therefore it is possible to invoke  Corollary~\ref{cor}.

\item[(ii)]
By Lemma~\ref{lema-cont}, the function given by \eqref{ex-rica} cannot be positive definite, since it is bounded and it does not coincide
with any continuous function a.e. However, $W$ is a nonnegative function
that belongs to $L^{1}(\R^N)\cap L^{\infty}(\R^N)$. Therefore Corollary~\ref{cor}
can be applied in any dimension.

\item[(iii)]
We recall that if $\Omega$ is  an even function, smooth away from the origin, homogeneous of
degree zero, with zero mean-value {on} the sphere
\bqq
\int_{\mathbb S^{N-1}}\Omega(\sigma)\,d\sigma =0,
\eqq
{then}
\bqq 
K(x)=\frac{\Omega(x)}{\abs{x}^N}, \quad x\in\R^N\backslash\{0\},
\eqq
 defines a tempered  distribution $\mathcal K$
in the sense of principal value, that coincides with $K$ away from the origin. Moreover, for any $f\in S(\R^N)$, $x\in \R^N$,
\bq\label{def-K}
(\mathcal K*f)(x)=\textrm{p.v.} \int_{\R^N}K(y)f(x-y)\,dy=\lim_{\ve\to 0} \int_{\frac1\ve>\abs{y}>\ve}\frac{\Omega(y)}{\abs{y}^N}f(x-y)\,dy,
\eq
 $\mathcal K \in \M_{p,p}(\R^N)$ for every $1<p<\infty$, and the Fourier transform of $\mathcal K$ belongs to $L^\infty(\R^N)$ (cf. \cite{stein}).
Therefore \bq\label{W-K} W=\alpha_1\delta+ \alpha_2\mathcal K\eq is a positive definite distribution
if $\alpha_1$ is large enough and then Theorem~\ref{global}-(ii) gives a global solution of \eqref{NGP}
in any dimension. For instance, we may consider in dimension three the function $K$ given by \eqref{dipolar}.
Since (see \cite{remi}) $$\wh{\mathcal K}(\xi)=\frac{4\pi}{3}\left(\frac{3\xi_3^2}{\abs{\xi}^2}-1\right), \qquad \xi\in\R^3\backslash\{0\},$$
\eqref{W-K} is positive definite by Proposition~\ref{lema-def} if
\bq\label{alpha}\alpha_1\geq \frac{4\pi}{3}\alpha_2\geq 0 \qquad\textrm{or}\qquad \alpha_1\geq -\frac{8\pi}{3}\alpha_2\geq 0.\eq
Therefore, if \eqref{alpha} is verified we may apply Theorem~\ref{global}-(i)-(a). Moreover,
if the inequalities in  \eqref{alpha} are strict, we have also the growth estimate of Theorem~\ref{global}-(ii).

\item[(iv)] Let us recall that to pass from the original equation \eqref{GP-full} to \eqref{GP-intro}
(and hence to \eqref{NGP}) we only need the constant $V*1$ be positive. If we take
$V$ as the potential given in the examples (i) or (ii), then $V\in L^1(\R^N)$ and
$$V*1=\int_{\R^N}V(x)\,dx>0.$$
Therefore Theorem~\ref{global} also provides the global well-posedness for the equation \eqref{GP-full}.
If we want to consider $V$ as in the example (iii),
the meaning of $\mathcal K*1$ is not obvious.
However, \eqref{def-K} still makes sense if $f\equiv 1$. In fact, using {\eqref{def-K}},
$$(\mathcal K*1)(x)=\lim_{\ve\to 0} 
\int_{\ve}^{\ve^{-1}}\int_{\mathbb S^2}\frac{\Omega(\sigma)}{r^3}r^2\,d\sigma \,dr=0.$$
Then if $V$ is given by \eqref{W-K}, $V*1=\alpha_1$  and we have the same conclusion as before,
provided that $\alpha_1>0$.
\end{enumerate}

One of the first works that introduces the nonlocal interaction in the Gross-Pitaevskii
equation was made by  Pomeau and Rica in \cite{rica0} considering
the potential \eqref{ex-rica}. Their main purpose was to establish a model
for superfluids with rotons. In fact, the Landau theory
of superfluidity of Helium II  says that
the dispersion curve must exhibit a roton minimum  (see \cite{landau,feyman})
as was  corroborated later by experimental observations (\cite{donnelly}).
Although the model considered in \cite{rica0} has a good  fit with the roton
minimum, it does not provide a correct sound speed. For this reason
 Berloff in \cite{berloff} proposes the potential
\bq\label{potential-berloff}
W(x)=(\alpha+\beta A^2\abs{x}^2+\gamma A^4\abs{x}^4)\exp(-A^2 \abs{x}^2), \quad x\in \R^3,
\eq
where the parameters $A$, $\alpha$, $\beta$ and $\gamma$  are chosen
such that the above requirements are satisfied.
However, the existence of this roton minimum implies that $\wh W$ must be
negative in some interval. In addition, a numerical
simulation in \cite{berloff} shows that in this case the solution exhibits
nonphysical mass concentration phenomenon, for certain initial conditions in $\phi+H^1(\R^3)$.
At some point, our results are in agreement with these observations in the sense that Theorem~\ref{global}
cannot be applied to the potential \eqref{potential-berloff}, because $\wh W$ and $W$ are
negative in some interval. However, by Proposition~\ref{regularidad} we may use the following local well-posedness result

\begin{teo}\label{teo-local}
Let $W$ be a distribution satisfying \eqref{Wn}.
Then  the Cauchy problem \eqref{NGP}
is locally well-posed in $\phi+H^1(\R^N)$. More precisely,
for every $w_0\in H^1(\R^N)$ there exists $T>0$ such that there
is a unique  $w \in C([-T, T],H^1(\R^N))$, for which $\phi+w$
solves \eqref{NGP} with the initial condition $u_0=\phi+w_0.$
In addition, $w$ is defined on a maximal time interval $(-T_{\min},T_{\max})$
where $w \in C^1((-T_{\min},T_{\max}),H^{-1}(\R^N))$ and
the blow-up alternative holds:  $\norm{w(t)}_{H^1(\R^N)}\to \infty$,
as $t\to T_{\max}$ if $T_{\max}<\infty$ and
$\norm{w(t)}_{H^1(\R^N)}\to \infty$,
as $t\to T_{\min}$ if $T_{\min}<\infty$.
Furthermore, supposing that $W$ is a real-valued even distribution,
for any {bounded} closed interval $I\subset (-T_{\min},T_{\max})$ the flow map $w_0\in H^1(\R^N) \mapsto w \in C(I,H^1(\R^N))$
is continuous and the energy and the generalized momentum are conserved on  $(-T_{\min},T_{\max})$.
\end{teo}

It is an open question to establish which are the exact implications of change of sign of the Fourier
transform of the potential for the global existence of the solutions of \eqref{NGP}.
As proposed in  \cite{berloff0}, a way to handle this problem would be to add a higher-order
nonlinear term in \eqref{GP-full}
to avoid the mass concentration phenomenon, maintaining the correct phonon-roton dispersion curve.

This paper is organized as follows. In the next section we give several results
about positive definite and positive distributions. In Section 3 we establish
some convolution inequalities that involve the hypothesis \eqref{Wn} and
we give the proof of Corollary~\ref{cor}. We prove the local well-posedness
in Section~\ref{sec-local} and also Propositions \ref{convergencia} and \ref{prop-H2}.
Theorem~\ref{global} is completed in Section~\ref{sec-global}.
In Section~\ref{gallo} we briefly recall the arguments that lead to Theorem~\ref{teo-gallo}
and in Section~\ref{conservation} we study the conservation of momentum and mass.

\section{Positive definite and positive distributions}\label{section-pos}
The purpose of this section is to recall some classical results for positive
definite and positive distributions, in the context
of Theorem~\ref{global}. We also state some properties
that we do not use in the next sections, but are
useful to better understand the type  of potentials considered
in Theorem~\ref{global}.

L. Schwartz in \cite{sc} defines that a (complex-valued) distribution $T$
is positive definite if
\bq\label{def-pos2}
 \langle T,\phi*\breve{\phi} \rangle \geq 0, \quad \forall \phi \in C_0^\infty(\R^N;\C),\eq
with $\breve{\phi}(x)=\overline{\phi}(-x)$. In virtue of our hypothesis on $W,$ we have preferred to adopt the simpler
definition \eqref{def-pos}. The relation between these two
possible definitions is given in the following lemma.
\begin{lema}\label{def-equiv} Let $T$ be a real-valued distribution.
\begin{itemize}
\item[$(i)$] If $T$ is positive definite \pare{in the sense of \eqref{def-pos}} and even, then  $T$
fulfils \eqref{def-pos2}.
\item[$(ii)$] If $T$ verifies \eqref{def-pos2}, then $T$ is even.
\end{itemize}
In particular, an even real-valued distribution is positive definite \pare{in the sense of \eqref{def-pos}} if and only if
it satisfies \eqref{def-pos2}.
\end{lema}
\begin{proof}
Suppose that $T$ is positive definite in the sense of \eqref{def-pos}.
Let $\phi \in C_0^\infty(\R^N;\C)$, with $\phi=\phi_1+ i \phi_2$,
$\phi_1,\phi_2 \in C_0^\infty(\R^N;\R)$. Then
\bq\label{dem1}
\langle T,\phi*\breve{\phi} \rangle =
\langle T,\phi_1*\widetilde{\phi}_1 \rangle+\langle T,\widetilde{\phi}_2*\phi_2 \rangle
+i\langle T,\widetilde{\phi}_1*{\phi_2}\rangle- i\langle T, \phi_1*\widetilde{\phi}_2 \rangle.
\eq
Since $W$ is even,
$$\langle T,\widetilde{\phi}_1*{\phi}_2\rangle=
\langle T,{{\phi}_1*\widetilde{\phi}_2}\rangle.$$
Therefore the imaginary part in the r.h.s. of \eqref{dem1} is zero.
The real part is positive because $T$ is positive definite, which
implies that $T$ verifies \eqref{def-pos2}.

For the proof of (ii), see \cite{sc}.
\end{proof}

The next result characterizes the positive definite distributions
under the hypotheses of Theorem~\ref{global}. In particular, it gives a simple
way to check the positive definiteness in terms of the Fourier transform.

\begin{prop}\label{lema-def}
Let $W\in \M_{2,2}(\R^N)$ be an even real-valued distribution. The following assertions are equivalent
\begin{itemize}
\item[$(i)$] $W$ is a positive definite distribution.
\item[$(ii)$] $\wh W\in L^\infty(\R^N)$ and $\wh W(\xi)\geq 0$ for almost every $\xi\in \R^N$.
\item[$(iii)$] For every $f\in L^2(\R^N;\R)$,
\bqq
\int_{\R^N}(W*f)(x)f(x)\,dx \geq 0.
\eqq
\end{itemize}
\end{prop}
\begin{proof}
{(i) $\Rightarrow$ (ii)}.  By Lemma~\ref{def-equiv}, we may apply the so-called Schwartz-Bochner
Theorem (see \cite{sc}, p.~276). Then there exists a positive measure $\mu\in S'(\R^N)$ such that
$\wh W=\mu$. Since $W\in \M_{2,2}(\R^N)$, we have that $\wh W\in L^\infty(\R^N)$,
and therefore $\wh W$ is a nonnegative bounded function.\\
{(ii) $\Rightarrow$ (iii)}. Since $W\in \M_{2,2}(\R^N)$, $W*f\in L^2(\R^N)$.
From the fact that $S(\R^N)$ is dense in $L^2(\R^N)$, we also have  that
$$\wh{W*f}=\wh W \wh f.$$
Using that $f$ is real-valued, by Parseval's theorem we finally deduce
$$ \int_{\R^N}(W*f)(x)f(x)\,dx=(2\pi)^{-N} \int_{\R^N}\wh W(\xi)\abs{\wh f(\xi)}^2\,d\xi\geq 0,
$$
where we have used that $\wh W\geq 0$ for the last inequality.\\
{(iii) $\Rightarrow$ (i)}. This implication directly follows from the fact that $C_0^{\infty}(\R^N;\R)\subset L^2(\R^N;\R)$.
\end{proof}

We remark that a positive definite distribution
is not necessarily a positive distribution.
For instance, we consider the Laguerre-Gaussian functions
\begin{equation}
W_m(x)=e^{-\abs{x}^2}\sum_{k=0}^m\dfrac{(-1)^k}{k!} {\binom{m+\frac{N}2}{m-k}} \abs{x}^{2k},  \quad x\in \R^N, \ m\in\N.
\label{laguerre}
\end{equation}
These functions are negative in some subset of $\R^N$ and since $\wh W_m\geq 0$ (see e.g.~\cite{fasshauer}, p.~38), Proposition~\ref{lema-def}
shows that they are positive definite functions. We also have that $W_m\in L^{1}(\R^N)\cap L^{\infty}(\R^N)$.
Then Corollary~\ref{cor} gives global existence of \eqref{NGP} for the potential \eqref{laguerre} in any dimension $N\geq 2$.

In the case {that} the considered distribution is actually a bounded function, its positive definiteness
gives some regularity. In other direction, the concept of positive definiteness may be related
to the same concept used for matrices. We recall some of these results in the next lemma.

 \begin{lema}\label{lema-cont} Let $W$ be an even real-valued positive definite distribution.
\begin{itemize}
\item[$(i)$] If $W\in L^\infty(\R^N)$,
then it coincides almost everywhere with
a continuous function.
\item[$(ii)$] If $W$ is continuous, then
$W(0)= \norm{W}_{L^\infty(\R^N)}$ and for all $x_1, \dots, x_m\in \R^N$, ${m\geq 1}$,
the matrix given by  ${A_{j k}=W(x_j-x_k)}$, ${j,k\in\{1,\dots,m\}}$, is a positive semi-definite matrix.
\end{itemize}
\end{lema}
\begin{proof}
Taking into consideration Lemma~\ref{def-equiv}, these statements are proved in \cite{sc}.
\end{proof}

The importance of the condition \eqref{hat-W} is that it gives the following coercivity property to the potential energy.
\begin{lema}\label{lema-L2}
Assume that  $W\in \M_{2,2}(\R^N)$ verifies \eqref{hat-W}. Then for all $f\in L^2(\R^N;\R)$,
\bq\label{L2-1}
 \sigma\LL{f}{2}^2\leq \int_{\R^N}(W*f)(x)f(x)\,dx \leq \norm{W}_{2,2} \LL{f}{2}^2. \eq
\end{lema}
\begin{proof}
The first inequality follows from Parseval's theorem,
$$\int_{\R^N}(W*f)(x)f(x)\,dx= (2\pi)^{-N}\int_{\R^N}\wh W(\xi)\abs{\wh f(\xi)}^2\,d\xi\geq \sigma\LL{f}2^2.$$
The second inequality in \eqref{L2-1} is immediate since $W\in \M_{2,2}(\R^N)$.
\end{proof}

The purpose of the last lemma in this section is to establish some properties of
the positive distributions which appear in Theorem~\ref{global}. In particular, we show that for these distributions
  \eqref{Wn} is automatically verified if $1\leq N\leq 3$.
\begin{lema}\label{lema-1-1} Let $W\in \M_{1,1}(\R^N)$ be a positive distribution. Then
$W\in \M_{p,p}(\R^N)$, for any $1\leq p\leq \infty$ and $W$ is a positive Borel
measure of finite mass. If $1\leq N\leq 3$ we also have that $W$ satisfies \eqref{Wn}.
\end{lema}
\begin{proof} Since  $W\in \M_{1,1}(\R^N)$, it is well known that
$W$ is a (complex-valued) finite Borel measure. Then
$W\in \M_{\infty,\infty}(\R^N)$ and by interpolation $W\in \M_{p,p}(\R^N)$
for any $1\leq p\leq \infty$.
Finally, the fact that $W$ is a positive distribution implies
that it is a positive measure (cf. \cite{sc}). By Proposition~\ref{regularidad} we conclude
that $W$ satisfies \eqref{Wn}, if $1\leq N\leq 3$.
\end{proof}
\section{Some consequences of assumption {(\ref{Wn})}}
We first establish some inequalities involving the convolution with $W$ that explain in part how
the hypothesis \eqref{Wn} works. After that, we give the proof of Proposition~\ref{regularidad} and Corollary~\ref{cor}.

From now on we adopt the standard notation $C(\cdot,\cdot, \dots )$ to represent a generic constant
that depends only on each of its arguments, and possibly on some fixed numbers such as the dimension.
In the case that $W\in \M_{p,q}(\R^N)$ we use $C(W)$ to denote a constant that only
depends on the norm $\norm{W}_{p,q}$. We also use the notation $p'$ for the 
conjugate exponent of $p$ given by $1/p+1/p'=1$.

\begin{lema}\label{young2}
Let $W\in \M_{p_1,q_1}(\R^N)\cap \M_{p_2,q_2}(\R^N)\cap \M_{p_3,q_3}(\R^N)$, with
\begin{equation*}
p_1, p_2,p_3,q_1,q_2,q_3\geq 1\quad \textup{and}\quad\frac{1}{p_3}+\frac{1}{q_2}=\frac{1}{q_1}.
\label{cond}
\end{equation*}
Suppose that there are $s_1,s_2\geq 1$, such that
\begin{equation*}
\frac{1}{p_1}-\frac{1}{p_3}=\frac{1}{s_1}, \quad \frac{1}{q_1}-\frac{1}{q_3}=\frac{1}{s_2}.
\end{equation*}
Then for any $u,v\in S(\R^N)$
\begin{align*}
\LL{(W*u)v}{q_1}&\leq{\norm{W}_{p_2,q_2}} \LL{u}{p_2}\LL{v}{p_3},\\
\LL{(W*u)v}{q_1}&\leq {\norm{W}_{p_3,q_3}}\LL{u}{p_3}\LL{v}{s_2},\\
\LL{W*(uv)}{q_1}&\leq {\norm{W}_{p_1,q_1}} \LL{u}{p_3}\LL{v}{s_1}.
\end{align*}
\end{lema}
\begin{proof}
The proof is a direct consequence of H\"older inequality and the hypotheses on $W$.
\end{proof}

\begin{lema}\label{last}
Assume that $W$ satisfies \eqref{Wn} {and that} $N\geq 4$. {Then} $W\in \M_{\frac{N}{N-2},2}(\R^N)$,
$W\in \M_{\frac{N}{N-2},\frac{N}2}(\R^N)$ and $W\in \M_{2,\frac{N}2}(\R^N)$.
\end{lema}
\begin{proof}
From the Riesz-Thorin interpolation theorem
and the fact that
$\left(\frac12,\frac2N \right)$ and $\left(\frac{N-2}{N},\frac{2}{N} \right)$ belong
to the convex hull of $$\left\{\left(\frac{1}{2},\frac12 \right),
\left(\frac1{p_1},\frac{1}{q_1}\right),\left(\frac1{p_3},\frac{1}{q_3}\right),\left(\frac1{p_4},\frac{1}{q_4}\right)\right\},$$
we conclude that  $W\in \M_{2,\frac{N}2}(\R^N)$ and
$W\in \M_{\frac{N}{N-2},\frac{N}2}(\R^N)$.  Since the conjugate exponent of $\frac{N}{N-2}$ is $\frac{N}{2}$,
$W\in \M_{2,\frac{N}2}(\R^N)$ implies that $W\in \M_{\frac{N}{N-2},2}(\R^N)$.
\end{proof}

\begin{lema}\label{young1}
Assume that $W$ satisfies \eqref{Wn}. Then for any $u,v,w\in S(\R^N)$,
\begin{align}\label{young-2}
\LL{(W*(uv))w}{\tilde  \gamma}&\leq C(W) \LL{u}{\tilde  s}\LL{v}{\tilde  r}\LL{w}{\tilde  r},
\end{align}
for some $2> \tilde \gamma>\frac{2N}{N+2}$,
 $\frac{2N}{N-2}>\tilde  r, \tilde s>2 $ if $N\geq 3$,
and $2> \tilde  \gamma>1$, $\infty> \tilde  r,\tilde s> 2$ if $N=1,2$.
\end{lema}

\begin{proof}
If $N\geq 4$, by Lemma~\ref{last} we have that $W\in \M_{\frac{N}{N-2},{\frac{N}2}}(\R^N)$.
Since also $W\in\M_{p_4,q_4}(\R^N)$, from the Riesz-Thorin interpolation theorem
we deduce that there exist $\bar p$ and $\bar q$ such {that}
\begin{equation}
W\in\M_{\bar p,\bar q}(\R^N), \quad \frac{N}{N-1}<\bar p< \frac{N}{N-2},\quad \frac{N}{2}<\bar q<N.
\label{p-q}
\end{equation}
Now we set
$$\frac{1}{\tilde r}=\min\left\{\frac12\left(1-\frac1{\bar q}\right),\frac1{2\bar p}\right\}, \ \frac{1}{\tilde \gamma}=\frac{1}{\bar q}+\frac{1}{\bar r},
\  \frac{1}{\tilde s}=\frac{1}{\bar p}-\frac{1}{\tilde  r}.$$
In view of \eqref{p-q}, we have  $\frac{2N}{N+2}<\tilde  \gamma< 2$ and $2<\tilde  r,\tilde  s<\frac{N-2}{2N}$.
By H\"older inequality, we conclude that
\begin{align*}
\LL{(W*(uv))w}{\tilde \gamma}&\leq \LL{W*(uv)}{\bar q}\LL{w}{\tilde  r}\\
&\leq \norm{W}_{\bar p,\bar q}\LL{uv}{\bar p}\LL{w}{\tilde  r}\\
&\leq \norm{W}_{\bar p,\bar q}\LL{u}{\tilde  s}\LL{v}{\tilde  r}\LL{w}{\tilde  r}.
	\end{align*}
If $N=1,2,3$, the proof is simpler. It is sufficient to take $\bar q=2$, $\bar p=2$, $\tilde s=\tilde r=4$, $\tilde \gamma=\frac43$ in the last inequality to deduce
\eqref{young-2}.
\end{proof}

 \begin{lema}\label{lema-int-finita}
Assume that $W$ satisfies \eqref{Wn}.
\item[$(i)$] For any $u\in \phi+H^1(\R^N)$ we have $(W*(1-\abs{u}^2))(1-\abs{u}^2)\in L^1(\R^N)$
\item[$(ii)$] If $W$ is also an even real-valued distribution,  then for any $u\in \phi+H^1(\R^N)$ and $h\in H^1(\R^N)$,
\bq \label{dem-3}
\int_{\R^N} (W*\llave{u, h})(1-\abs{u}^2)\,dx=\int_{\R^N} (W*(1-\abs{u}^2))\llave{u, h}\,dx.
\eq
\end{lema}
\begin{proof}
Let $u=\phi+w$, with $w\in H^1(\R^N)$.
If $N\geq 4$, by \eqref{phi} and the Sobolev embedding theorem, we deduce that
\bqq (1-\abs{\phi}^2-2\llave{\phi, w}-\abs{w}^2)\in L^2(\R^N)+L^\frac{N}{N-2}(\R^N).\eqq
By Lemma~\ref{last} we have that the map
$h\mapsto W*h$ is continuous from $L^2(\R^N)+L^\frac{N}{N-2}(\R^N)$ to $L^2(\R^N)\cap L^\frac{N}{2}(\R^N)$
and since $\frac{N-2}{N}+\frac{2}{N}=1$, by H\"older inequality we conclude that
\bq\label{dem-w-2} (W*(1-\abs{\phi}^2-2\llave{\phi, w}-\abs{w}^2))(1-\abs{\phi}^2-2\llave{\phi, w}-\abs{w}^2) \in L^1(\R^N).\eq
If $1\leq N\leq 3$, \eqref{dem-w-2} follows from the fact that $\abs{w}^2\in L^2(\R^N)$. This concludes the proof of (i).

A similar argument shows that $\LL{(W*\llave{u,h})(1-\abs{u}^2)}{1}<\infty$. Then
using that $W$ is even and  Fubini's theorem we obtain (ii).
\end{proof}

The previous lemmas will be useful in the next sections, in particular to prove
the local well-posedness of \eqref{NGP}. Now we give the
proofs of Proposition~\ref{regularidad} and Corollary~\ref{cor},
that involve some straightforward computations.

\begin{proof}[Proof of Proposition~\ref{regularidad}]
For the first part of (i), we note that the hypothesis
implies that $W\in \M_{p,p}(\R^N)$ for any $\frac32\leq p\leq 3$. Then it is sufficient  to take
$p_1=q_1=\frac{3}{2}$, $p_2=p_3=q_2=q_3=3$ 
and $p_4=q_4=2$ to see that \eqref{Wn} is fulfilled.
For the second part of (i), we need prove that $W\in \M_{3,3}(\R^N)$. 
Recalling that $\M_{p,q}(\R^N)=\M_{q',p'}(\R^N)$ for $1<p\leq q<\infty$
and using the Riesz interpolation
theorem, we have that $W\in \M_{s,t}(\R^N)$, for every $(s^{-1},t^{-1})$ in the convex hull of
\bq\label{convex-set}
\left\{\left(\frac12,\frac12\right)\right\}\cup 
\bigcup_{j=1}^3 \left\{\left(\frac1{p_j},\frac1{q_j}\right),\left(1-\frac1{q_j},1-\frac1{p_j}\right)  \right\}.
\eq
By hypothesis, $p_i=q_i$, $i=1,2,3$, thus \eqref{Wn} implies that
\bqq
{\frac{1}{p_2}+\frac{1}{p_3}=\frac{1}{p_1}, \quad 2\geq p_1 \textup{ and } \ p_2,p_3\geq 2.}
\eqq
Hence the convex hull of \eqref{convex-set} simplifies to
\bqq
{\left\{(x,x) \in \R^2 : \min\left\{1-\frac1{p_1},\frac1{p_2},\frac1{p_1}-\frac1{p_2}\right\}\leq x\leq 
\max\left\{\frac1{p_1},1-\frac1{p_2},1-\frac1{p_1}+\frac1{p_2}\right\}
\right\}.}
\eqq
Arguing by contradiction, it is simple to see that
\bqq{\min\left\{1-\frac1{p_1},\frac1{p_2},\frac1{p_1}-\frac1{p_2}\right\}\leq \frac 13 \quad \textrm{and}\quad 
\frac 23\leq \max\left\{\frac1{p_1},1-\frac1{p_2},1-\frac1{p_1}+\frac1{p_2}\right\}.}\eqq
{Therefore} $W\in \M_{s,s}(\R^N)$, for every $\frac32 \leq s \leq 3$. In particular $W\in \M_{2,2}(\R^N)\cap \M_{3,3}(\R^N)$.


To prove (ii), we notice that by
interpolation we have that
$W\in \M_{\alpha,\beta}(\R^N)$, for all
$\alpha,\beta$ satisfying
\begin{equation}\label{dem-alpha}
1\leq \alpha,\beta, \quad \frac{1}{\alpha} -\left(1-\frac{1}{\bar r}\right)\leq \frac{1}{\beta}\leq \frac{1}{\alpha}.
\end{equation}
We now define
\begin{align*}
p_2=p_3=&\begin{cases}
3, & \textrm{ if } 4\leq N\leq 5,\\
\frac{ s_N}{ s_N-1}, & \textrm{ if } 6 \leq N,
\end{cases}
&\phantom{A}q_2=q_3=\begin{cases}
{3}, & \textrm{ if } 4\leq N\leq 5,\phantom{ooi}\\
N, & \textrm{ if } 6 \leq N,
\end{cases}\\
p_1=&\begin{cases}
\frac{3}{2}, & \textrm{ if } 4\leq N\leq 5,\\
\frac{N}{N-1}, & \textrm{ if } 6 \leq N,
\end{cases}
&\phantom{A}q_1=\begin{cases}
\frac{3}{2}, & \textrm{ if } 4\leq N\leq 5,\\
\frac{p_3q_2}{p_3+q_2}, & \textrm{ if } 6 \leq N,
\end{cases}
\end{align*}
$p_4=\frac{2\bar r}{2\bar r-1}$, $q_4=2\bar r$,
where $$s_N=\begin{cases}
\dfrac{N}{4}+\ve_N, & \textrm{ if } 6\leq N\leq 7,\\
\dfrac{2(N+1)}{N+2}, &  \textrm{ if } 8\leq N,
\end{cases}$$
and $\ve_N>0$ is chosen small enough such that $0<\ve_N<2-\frac{N}{4}$ if $6 \leq N\leq 7$.
Then we have that
\begin{equation}\label{range}
\frac{2N}{N+2}<s_N<2, \quad \textrm{ for any } N\geq 6.
\end{equation}
Using that $\bar r>\frac{N}{4}$ and \eqref{range}, we can verify that the choice of $(p_i,q_i)$, $i\in\{1,\dots,4\}$,
satisfies \eqref{dem-alpha} with $\alpha=p_i$ and $\beta=q_i$,
as well as all the others restrictions in the hypothesis \eqref{Wn}, which completes the proof.
\end{proof}

\begin{proof}[Proof of Corollary~\ref{cor}]
By Young inequality we have that
$W\in\M_{p,p}(\R^N)$, for any $1\leq p\leq \infty$.
In particular the condition $W\in\M_{1,1}(\R^N)$ is fulfilled.
If $1\leq N\leq 3$, the conclusion is a consequence of
Proposition~\ref{regularidad} and Theorem~\ref{global}.
If $N\geq 4$, by
Young inequality we have that
$W\in \M_{p,q}(\R^N)$, for all $1-\frac{1}{r}\leq \frac{1}p\leq 1$,
with $\frac{1}{q}=\frac1p+\frac1{r}-1$.
Then the proof follows again from
Proposition~\ref{regularidad} and Theorem~\ref{global}.
\end{proof}

\section{Local existence}\label{sec-local}
In order to prove Theorem~\ref{global} we
first are going to prove the local well-posedness. Theorem~\ref{teo-local} is based on the fact that if we set $u=w+\phi$, then $u$ is a solution of \eqref{NGP}
with initial condition $u_0=\phi+w_0$
if and only if $w$ solves
\begin{ecu} \label{GP1}
i \ptl _t  w+\Delta w+f(w)&=0 \textrm{ on } \R^N\times \R,\\
w(0)&=w_0,
\end{ecu}
with
\bq
f(w)=\Delta\phi+(w+\phi)(W*(1-\abs{\phi+w}^2)). \nonumber
\eq
We decompose $f$ as
\bq\label{f}
f(w)=g_1(w)+g_2(w)+g_3(w)+g_4(w),
\eq
with
\begin{align*}
g_1(w)&=\Delta \phi+(W*(1-\abs\phi ^2))\phi,\\
g_2(w)&=-2(W* \llave{\phi, w})\phi,\\
g_3(w)&=-(W*\abs w^2)\phi-2(W* \llave{\phi, w})w+(W*(1-\abs \phi^2))w,  \\
g_4(w)&=-(W*\abs w^2)w.
\end{align*}
The next lemma gives some estimates on each of these functions.
\begin{lema}\label{cotas-g}
{Assume that} $W$ {satisfies \eqref{Wn}}. Using the numbers
given by \eqref{Wn} and Lemma~\ref{young1}, let $r_1=r_2=2$, 
$r_3=p_3$, $r_4=\tilde r$, $\rho_1=\rho_2=2$, $\rho_3=q_1'$ and 
$\rho_4=\tilde \gamma'.$ Then
\begin{equation}
g_j\in C(H^1(\R^N),H^{-1}(\R^N)), \ j\in\{1,2,3,4\}.
\label{g-cont}
\end{equation}
Furthermore, for any $M>0$ there exists a constant $C(M,W,\phi)$ such that
\bq \label{lip}
\LL{g_j(w_1)-g_j(w_2)}{\rho_j'}\leq C(M,W,\phi)\LL{w_1-w_2}{{r_j\vphantom{\rho_j'}}},
\eq
for all $w_1,w_2\in H^1(\R^N)$ with $\norm{w_1}_{H^1},\norm{w_2}_{H^1}\leq M$,
and
\begin{align}
\norm{g_j(w)}_{W^{1,\rho_j'}}&\leq C(M,W,\phi)(1+\norm{w}_{W^{1,r_j}\vphantom{W^{1,\rho_j'}}}), \label{g-grad}
\end{align}
for all $w\in H^1(\R^N)\cap {W^{1,r_j}(\R^N)}$ with $\norm{w}_{H^1}\leq M$.
\end{lema}

\begin{proof}
Since $g_1$ is a constant function of $w$, $g_1\in C(H^1(\R^N),H^{-1}(\R^N))$
and \eqref{lip} is trivial in this case.  The condition \eqref{g-grad} follows from
the estimate
\begin{align*}
\norm{g_1(w)}_{H^1}\leq& \norm{\nabla \phi}_{H^2}+\norm{W}_{2,2}
( \LL{1-\abs \phi^2}{2}\norm{\phi}_{W^{1,\infty}}+2\LL{\phi}{\infty}^2\LL{\nabla \phi}{2}).
\end{align*}
Similarly we obtain for $g_2$,
$$\LL{g_2(w_1)-g_2(w_2)}{2}\leq 2\norm{W}_{2,2}\LL{\phi}{\infty}^2 \LL{w_1-w_2}{2}$$
and
\begin{align*}
\LL{\grad g_2(w)}{2}&\leq 2\norm{W}_{2,2}\LL{\phi}{\infty}
\big(
\LL{\phi}{\infty} \LL{\grad w}{2}+2\LL{\grad\phi}{\infty} \LL{w}{2} \big)\\
&\leq C(W,\phi) \HH{w}.
\end{align*}
Then we deduce \eqref{lip} and \eqref{g-grad} for $j=2$.

For $g_3$, we have
\begin{multline}
g_3(w_2)-g_3(w_1)=(W*(\abs{w_1}^2-\abs{w_2}^2))\phi+2(W*\llave{\phi,w_1-w_2 })w_1\nonumber\\
+2(W*\llave{\phi, w_2})(w_1-w_2)+(W*(1-\abs{\phi}^2))(w_1-w_2).
\end{multline}
{The assumption \eqref{Wn} allows to apply Lemma~\ref{young2} and then
we derive}
\begin{equation}\label{eq:g3-lip}
\begin{split}
\LL{g_3(w_2)-g_3(w_1)}{\rho_3'}\leq  C(W,\phi)  \LL{{w_1}-{w_2}}{r_3}(\LL{w_1}{s_1}+\LL{w_2}{s_1} \\
{+2\LL{w_1}{s_2}+2\LL{w_2}{p_2}+1).}
\end{split}
\end{equation}
More precisely, the dependence on $\phi$ of the constant  $C(W,\phi)$ in the last inequality 
is given explicitly by ${\max\{\LL{\phi}{\infty},\LL{1-\abs{\phi}^2}{p_2}\}}$. By {the} Sobolev embedding theorem
\begin{equation}\label{inyec}
H^1(\R^N)\hookrightarrow L^p(\R^N), \quad \forall\,p\in \left[2,\frac{2N}{N-2}\right] \text{ if }N\geq 3  \text{ and } \forall\,p\in [2,\infty) \textrm{ if }N=1,2.
\end{equation}
In particular, $${\LL{w_1}{s_1}+\LL{w_2}{s_1}+ 2\LL{w_1}{s_2}+2\LL{w_2}{p_2}\leq C  (\HH{w_1}+ \HH{w_2}),}$$
which together with \eqref{eq:g3-lip} gives us \eqref{lip}
for $g_3$. With the same type of computations, taking $w\in H^1(\R^N)$, $\HH{w}\leq M$, we have
\begin{align*}
\LL{\grad g_3(w)}{\rho_3'}\leq&C(M,W,\phi)( \LL{\grad w}{r_3}+\LL{w}{r_3}),
\end{align*}
where the dependence on $\phi$ is in terms of $\LL{\phi}{\infty}$, $\LL{\grad \phi}{\infty}$,
$\LL{1-\abs{\phi}^2}{p_2}$ and $\LL{\grad \phi}{p_2}$.

For $g_4$, {applying  Lemma~\ref{young1} we obtain}
\bqq\begin{split}
 \LL{g_4(w_1)-g_4(w_2)}{\rho_4'}\leq C(W) \LL{w_1-w_2}{r_4}( (\LL{w_1}{s}+\LL{w_2}{s})\LL{w_1}{r_4}\\
+\LL{w_2}{s}\LL{w_2}{r_4})
\end{split}
\eqq
and
\begin{align*}
\LL{\grad g_4(w)}{\rho_4'}\leq& C(W)  \LL{\grad w}{r_4} \LL{w}{r_4} \LL{w}{s}.
\end{align*}
As before, using \eqref{inyec}, we conclude that $g_4$ verifies \eqref{lip}-\eqref{g-grad}.

Since for $2\leq j\leq 4$, $2\leq r_j<\frac{2N}{N-2}$ ($2\leq r_j<\infty$ if $N=1,2$), we have the continuous embeddings
$$H^1(\R^N)\hookrightarrow L^{r_j}(\R^N) \  \textrm{  and  } \ L^{r'_j}(\R^N)\hookrightarrow H^{-1}(\R^N).$$
Then inequality \eqref{lip} implies \eqref{g-cont}, for $j\in\{2,3,4\}$.
\end{proof}

Now we analyze the potential energy associated to \eqref{GP1}. For any $v\in H^1(\R^N)$
we set
\begin{equation}
F(v):=\int_{\R^N}\llave{ \Delta\phi, v}\,dx -\frac14 \int_{\R^N}(W*(1-\abs{\phi+v}^2))(1-\abs{\phi+v}^2)\,dx ,
\label{F}
\end{equation}
and using the notation of Lemma~\ref{cotas-g}, we fix for the rest of this section
\bq\label{r}
r=\max\{r_1,r_2,r_3,r_4,\rho_1,\rho_2,\rho_3,\rho_4 \}.
\eq
\begin{lema}\label{lema-dif-F} Assume that $W$ satisfies \eqref{Wn}.
Then the functional $F$ is well-defined on $H^1(\R^N)$.
If {moreover} $W$ is a real-valued even distribution, we have the following properties.
\begin{enumerate}
\item[$(i)$] $F$ is Fr\'echet-differentiable  and  \begin{equation}
F\in C^1(H^1(\R^N),\R) \text{  with  } F'=f. \label{der-F}
\end{equation}
\item[$(ii)$] For any $M>0$, there exists a constant $C(M,W,\phi)$ such that
\bq\label{diff-F}
\abs{F(u)-F(v)}\leq C(M,W,\phi)(\LL{u-v}{2}+\LL{u-v}{r}),
\eq
for any $u,v\in H^1(\R^N)$, with $\HH{u},\HH{v}\leq M$.
\end{enumerate}
\end{lema}
\begin{proof}

By Lemma \ref{lema-int-finita}, $F$ is well-defined in $H^1(\R^N)$ for any $N$.
To prove (i), we compute now the {G\^ateaux} derivative of $F$. For $h\in H^1(\R^N)$ we have
\begin{align*}
d_{G}F(v)[h]=&\lim_{t\to 0}\frac{F(v+th)-F(v)}{t}\\
=&\int_{\R^N} \llave{\Delta \phi, h}\,dx +\frac12\int_{\R^N}(W* \llave{\phi+v, h})(1-\abs{\phi+v}^2)\,dx\\
&+\frac12\int_{\R^N}(W*(1-\abs{\phi+v}^2)) \llave{\phi+v, h}\,dx.
\end{align*}
Since $W$ is an even distribution,  \eqref{dem-3}
implies that the last two integrals  are equal. Finally we get that
$$d_{G}F(v)[h]=\int_{\R^N}\llave{f(v), h}\,dx=\left\langle f(v),h\right\rangle_{H^{-1},H^{1}}.$$
From \eqref{f} and \eqref{g-cont}, we have that $f\in C(H^1(\R^N),H^{-1}(\R^N))$.
Hence the map $v\to d_{G}F(v)$ is continuous from $H^1(\R^N)$ to $H^{-1}(\R^N)$, which implies that $F$ is continuously Fr\'echet-differentiable
and satisfies \eqref{der-F}.

For the proof of (ii), using \eqref{der-F} and the mean-value theorem, we have
\begin{align*}
F(u)-F(v)&=\int_0^1 \frac{d}{ds}F(su+(1-s)v)\,ds =\int_0^1 \langle f(su+(1-s)v),u-v \rangle_{H^{-1},H^{1}}\,ds.
\end{align*}
Then by Lemma~\ref{cotas-g},
\begin{ecu0}\label{dem-dif}
\abs{F(u)-F(v)}&\leq\ \sup_{s\in[0,1]} \sum_{j=1}^4 \LL{g_j(su+(1-s)v)}{\rho_j'} \LL{u-v}{\rho_j}\\
&\leq \sum_{j=1}^4 C(M,W,\phi)(\LL{u}{r_j}+\LL{v}{r_j}+1)\LL{u-v}{\rho_j}.
\end{ecu0}
Since we assume that  $\HH{u},\HH{v}\leq M$, \eqref{inyec} implies that
\bq\LL{u}{r_j}+\LL{v}{r_j}+1\leq C(M).\label{dem-diff2}\eq
Also,  it follows from $L^p$-interpolation {and Young's inequality} that
\bq\label{inter}
\LL{u-v}{\rho_j}\leq \LL{u-v}{2}^{\theta_j}\LL{u-v}{r}^{1-\theta_j}\leq \LL{u-v}{2}+\LL{u-v}{r},
\eq
with $\theta_j=\frac{2(r-\rho_j)}{\rho_j(r-2)}.$
By combining \eqref{dem-dif}, \eqref{dem-diff2} and \eqref{inter}, we obtain (ii).
\end{proof}

\begin{proof}[Proof of Theorem~\ref{teo-local}]
Recalling that $r$ was fixed in \eqref{r}, we define $q$ by
$\frac{1}{q}=\frac{N}{2}\left(\frac{1}2-\frac1r\right)$. 
Given $T,M>0$, we consider the complete metric space
\bqq\begin{split}
X_{T,M}=\{w\in L^\infty((-T,T),H^1(\R^N))\cap  L^q((-T,T),W^{1,r}(\R^N)):  \\
\norm{w}_{L^\infty((-T,T),H^1)}\leq M, \ \norm{w}_{L^q((-T,T),{W^{1,r}})}\leq M \},
\end{split}\eqq
endowed with the distance
\bq\label{dist} d_{T}(w_1,w_2)=\norm{w_1-w_2}_{L^\infty((-T,T),L^2)}+\norm{w_1-w_2}_{L^q((-T,T),L^r)}.\eq
The estimates given in Lemmas~\ref{cotas-g},~\ref{lema-dif-F} and the Strichartz estimates
show that the functional
\bqq
\Phi(w)=e^{it\Delta}w_0+i\int_{0}^te^{i(t-s)\Delta}f(w(s))\,ds
\eqq
is a contraction in $X_{T,M}$ for some
$M\leq C(\HH{w_0}+1)$ and $T$ small enough, but depending
only on $\HH{w_0}$. Then we have a solution given by Banach's fixed-point theorem.
The arguments  to complete Theorem~\ref{teo-local}
 are rather standard. For instance, Theorem 4.4.6 in \cite{cazenave}
 automatically implies the existence, uniqueness, the blow-up alternative and that the function
$L(t)$ given by
\begin{equation*}
L(t):=L_1(t)+\frac14 \int_{\R^N}(W*(1-\abs{\phi+w(t)}^2))(1-\abs{\phi+w(t)}^2)\,dx ,
\end{equation*}
with $$L_1(t)=\frac12 \int_{\R^N}\abs{\grad{w(t)}}^2\,dx-\int_{\R^N} \llave{\Delta\phi ,w(t)}\,dx,$$
is constant for all $t\in(-T_{\min},T_{\max})$. Noticing that
$$L_1(t)=\frac12 \int_{\R^N}\abs{\grad{w(t)}+\grad\phi}^2\,dx -\frac12 \int_{\R^N}\abs{\grad{\phi}}^2\,dx,$$
we conclude that the energy is conserved.

However, the continuous dependence on the initial data in $H^1(\R^N)$
is not obvious, because the distance \eqref{dist} does not
involve derivatives. Therefore we give the complete proof of this point. Here we will omit the dependence on $W$ and $\phi$ {in the generic constant} $C$, since it plays no role
in the analysis of continuous dependence. Let $w_{0,n},w_0\in H^1(\R^N)$ be such that
$$w_{0,n}\to w_0 \quad \textrm{ in } H^1(\R^N).$$
  Then for some $n_0\geq 0$,
 \bqq\label{wn}\HH{w_{0,n}}\leq \HH{w_0}+1,\quad \forall n\geq n_0.\eqq
We denote $w_n$ and $w$ the solutions with initial data
$w_{0,n}$ and $w_0$, respectively. Then by the fixed-point argument, 
there exist $T>0$ and a constant $C(\HH{w_0})$, both depending only on $\HH{w_0}$, such that $w_n$ and $w$ are defined in $[-T,T]$
for all $n\geq n_0$ and  
\bq\label{cota-pedida}
\norm{w_n}_{L^\infty((-T,T),H^1)}+\norm{w}_{L^\infty((-T,T),H^1)} \leq C(\HH{w_0}), \quad \forall n\geq n_0.
 \eq
{Since}
\bqq
w_n(t)-w(t)=e^{it\Delta}(w_{0,n}-w_0)+i\int_{0}^t e^{i(t-s)\Delta}(f(w_n(s))-f(w(s)))\,ds,
\eqq
using Strichartz estimates we have that
\bq\label{stri}
d_T(w_n,w)\leq C\LL{w_{0,n}-w_0}{2}+C\sum_{j=1}^4 \norm{g_j(w_n)-g_j(w)}_{L^{\gamma_j'}{((-T,T),L^{\rho_j'})}},
\eq
with $\frac{1}{\gamma_j}=\frac{N}{2}\left( \frac12-\frac1{\rho_j} \right)$.
{By Lemma~\ref{cotas-g}, \eqref{cota-pedida}, using as in \eqref{inter} an} $L^p$-interpolation inequality and Young's inequality,  we deduce that
\bq\label{dem-lip}
\norm{g_j(w_n)-g_j(w)}_{\rho_j'} \leq C(\norm{w_0}_{H^1})(\LL{w_n-w}{2}+\LL{w_n-w}{r}).
\eq
 Applying H\"older inequality with ${\beta_j}=\frac{1}{\gamma_j'}-\frac{1}{q}$,
\bq\label{hol}
\norm{w_n-w}_{L^{\gamma_j'}{((-T,T),L^{r})}}\leq \norm{w_n-w}_{L^{q}{((-T,T),L^{r})}}(2T)^{\beta_j}.
\eq
Notice that $0<\beta_j\leq 1$ since $2\leq \rho_j,r_j<\frac{2N}{N-2}$.
Assuming $T\leq 1$ and putting together \eqref{dem-lip} and \eqref{hol} we conclude that
\bq\label{dem-stri}
\norm{g_j(w_n)-g_j(w)}_{L^{\gamma_j'}{((-T,T),L^{\rho_j'})}}\leq C(\norm{w_0}_{H^1}) T^{{\beta}}d_T(w_n,w),
\eq
with ${\beta}=\min\{{\beta_j,{1/{\gamma_j'}}}:\ 1\leq j\leq 4\}$. Choosing $T$ such that $4T^\beta C(\norm{w_0}_{H^1})\leq\frac12$,
\eqref{stri} and  \eqref{dem-stri} give
\bqq
d_T(w_n,w)\leq 2C(\norm{w_0}_{H^1})\HH{{w_{0,n}}-w_0}.
\eqq
Hence
\bqq
w_n\to w, \textrm{  in  } C([-T,T],L^2(\R^N)) \cap L^q((-T,T),L^r(\R^N)).
\eqq
Thus from \eqref{cota-pedida} and the
Gagliardo-Nirenberg inequality, we conclude that $w_n\to w$  in  $C([-T,T],L^p(\R^N))$,
for every $2\leq p <\infty$ if $N=1,2$ and $2\leq p<\frac{2N}{N-2}$ if $N\geq 3$.
Using the inequality \eqref{diff-F} in Lemma~\ref{lema-dif-F}, 
it follows that $F(w_n)\to F(w)$ in ${C([-T,T])}$.
Since the energy is conserved for $w$ and $w_n$, {this implies that}
\bqq
\LL{\grad w_n}{2}\to \LL{\grad w}{2} \ \textrm{  in  }\ C([-T,T]).
\eqq
In addition, from the equation $i\partial_t w_n=-\Delta w_n-f(w_n)$ in $[-T,T]$, we get 
$$\norm{\partial_t w_n}_{H^{-1}}\leq \norm{w_n}_{H^1}+\sum_{j=1}^4\norm{g_j(w_n)}_{H^{-1}},$$
Hence Lemma~\ref{cotas-g} and \eqref{cota-pedida} provide a uniform bound
for $w_n$ in $C^1([-T,T],H^{-1}(\R^N))$. Therefore
$w_n \to w$ in $C([-T,T],H^1(\R^N))$ (see Proposition 1.3.14 in \cite{cazenave}).
A covering argument allows us to finish the proof in any {closed} bounded interval.

Since the generalized momentum still needs a precise definition, we will postpone
the proof of its conservation until Section~\ref{conservation}.
\end{proof}\vspace{-0.2cm}
We prove now Propositions~\ref{convergencia} and \ref{prop-H2} because
the arguments involved are very similar to those used in this section.
For these proofs we suppose that Theorem~\ref{global} is already proved.

\begin{proof}[Proof of Proposition \ref{convergencia}] Let $u_n=\phi+w_n$ and $u_\infty=\phi+w_\infty$,
{where} $w_n,w_\infty\in C(\R,H^1(\R^N))$, be the global solution of
\eqref{NGP} with potentials $W_n$ and $W_\infty$, respectively,
with {the same} initial data $ u_0=\phi+w_0$, {with} $w_0\in H^1(\R^N)$.
In the same spirit of the proof of Theorem~\ref{teo-local},
for $v\in H^1(\R^N)$, we set
\bqq
f_n(v)=g_{1,n}(v)+g_{2,n}(v)+g_{3,n}(v)+g_{4,n}(v),
\eqq
with
\begin{align*}
g_{1,n}(v)&=\Delta \phi+(W_n*(1-\abs\phi ^2))\phi,\\
g_{2,n}(v)&=-2(W_n* \llave{\phi, v})\phi,\\
g_{3,n}(v)&=-(W_n*\abs v^2)\phi-2(W_n* \llave{\phi, v})w+(W_n*(1-\abs \phi^2))v,  \\
g_{4,n}(v)&=-(W_n*\abs v^2)v,
\end{align*}
for any $n\in \N\cup \{\infty\}$. 
Noticing that for any $v_1,v_2\in H^1(\R^N)$, $1\leq j\leq 4$,
$$g_{j,n}(v_1)-g_{j,m}(v_2)=\left(g_{j,n}(v_1)-g_{j,n}(v_2)\right)+\left(g_{j,n}(v_2)-g_{j,m}(v_2)\right),$$
{Proposition~\ref{regularidad}, Lemma~\ref{young2}, the proof of Lemma~\ref{young1}
and the same argument given in Lemma~\ref{cotas-g} allows} us to conclude that ({we omit from now on the dependence on }$\phi$)
\bq\label{dif-con}
\LL{g_{j,n}(v_1)-g_{j,m}(v_2)}{\rho_j'}\leq C(W_n,M)\LL{v_1-v_2}{r_j}+
C(W_n-W_m,M)(\LL{v_2}{r_j}{+1}),
\eq
for any $n,m\in \N\cup\{\infty\}$ and $v_1,v_2\in H^1(\R^N)$ with $\HH{v_1},\HH{v_2}\leq M$,
with (the new choice of) $\rho_j$, $r_j$ given by
\bq\label{rs} \rho_1=\rho_2=r_1=r_2=2, \ \rho_3=r_3=3, \  \rho_4=r_4=4, \eq
and \bq\label{def-sigma} {C(W,M)=\sigma(W)C(M)}, \textup{ with }\sigma(W)=\max \{\norm{W}_{2,2},\norm{W}_{3,3}\}.\eq
By the uniqueness provided by Theorem~\ref{global}, the functions
$w_n$ are given by the fixed-point argument
of the proof of Theorem~\ref{teo-local}. Since the
 estimates for the fixed point can be  obtained
using Lemma~\ref{cotas-g}, but with the values in \eqref{rs},
and by \eqref{conver}
we may assume that for $k=2,3$
$${\frac12\norm{W_\infty}_{k,k}\leq \norm{W_n}_{k,k}\leq  2\norm{W_\infty}_{k,k},}$$
so that we have uniform bounds on $W_n$. Therefore we conclude that
there exist some $T\leq 1$ and $C>0$ that only depend on  $\HH{w_0}$, $\norm{W_\infty}_{2,2}$
and $\norm{W_\infty}_{3,3}$
such that
\bq\label{cota-w} \norm{w_n}_{L^\infty((-T,T),H^1)}\leq {C},\quad \textup{ for any }n\in \N\cup \{\infty\}.\eq
Using the distance
\bqq d_{T}(w_1,w_2)=\norm{w_1-w_2}_{L^\infty((-T,T),L^2)}+\norm{w_1-w_2}_{L^{\frac{8}{N}}((-T,T),L^{4})},\eqq
the estimates \eqref{dif-con}, \eqref{cota-w} and following the lines of the proof of Theorem~\ref{teo-local},
it leads {to}
\bqq d_T(w_n,w_\infty)\leq C \sigma(W_n-W_\infty).\eqq
Hence the hypothesis \eqref{conver} and \eqref{def-sigma} imply that
\bqq
w_n\to w_\infty \quad \textrm{  in  } C([-T,T],L^2(\R^N)) \cap L^\frac{8}{N}((-T,T),L^{4}(\R^N)).
\eqq
Then \eqref{cota-w} and the Gagliardo-Nirenberg inequality imply
that
\bq\label{con-w2}
w_n\to w_\infty\ \textrm{  in  }C([-T,T],L^p(\R^N)), \ \forall\, p\in [2,\infty)\textrm{ if } N=1,2 \textrm{ and } \forall\,p\in\left[2,\frac{2N}{N-2}\right)
\textrm{ if } N\geq 3.
\eq

We denote by $F_n$ the function given by \eqref{F}, with $W$ replaced by $W_n$,
so that the conserved energy for each $u_n$ is
\bq\label{energy} E_n(u_n(t))=\LL{\grad w_n(t)}{2}+F_n(w_n(t))=\LL{\grad w_0}{2}+F_n(w_0),\quad \textrm{for any }  t \in\R.\eq
The inequality \eqref{dif-con} and similar arguments as in the proof of Lemma~\ref{lema-dif-F} give for any $v_1,v_2\in H^1(\R^N)$ with ${\HH{v_1},\HH{v_2}\leq M}$, that
there exists a constant $C$ depending only on $M$, $\norm{W_\infty}_{2,2}$
and $\norm{W_\infty}_{3,3}$, such that
\bq\label{dif-F2} \abs{F_n(v_1)-F_m(v_2)}\leq C\left(\LL{v_1-v_2}{2}+\LL{v_1-v_2}{4}\right)+C\sigma(W_n-W_m).\eq
By putting together \eqref{cota-w}, \eqref{con-w2} and  \eqref{dif-F2}, we deduce that $F_n(w_n)\to F_\infty(w_\infty)$
in ${C([-T,T])}$. Then by \eqref{energy} we have that $\LL{\grad w_n}{2}\to \LL{\grad w_\infty}{2}$ in  $C([-T,T])$.
The conclusion follows as in the proof of Theorem~\ref{teo-local}.
\end{proof}
\begin{proof}[Proof of Proposition \ref{prop-H2}] Using the notation
introduced at the beginning of this section, by Lemma 5.3.1 in \cite{cazenave}, we only
need to prove that for any $1\leq j\leq 4$ and any $w\in H^{s}(\R^N)$ such that $\HH{w}\leq M,$ we have
\bq\label{g-H2} \LL{g_j(w)}{2}\leq C(W,M,\phi)\left(1+\norm{w}_{H^s}\right),\eq
for some $0<s<2$. From the estimate \eqref{g-grad} in Lemma~\ref{cotas-g}
and the Sobolev embedding theorem, we have the inequality \eqref{g-H2} for $j=1,2$
for any $s\geq 1.$  For $j=3,4$ we note that by the Sobolev embedding theorem,
\bqq
W^{1,p}(\R^N)\hookrightarrow L^2(\R^N), \ \forall p\in \left[\frac{2N}{N+2},2\right] \textrm{ if }N\geq 3
\ \textrm{and }\ \forall p\in[1,2] \textrm{ if }N=1,2\textrm,
\eqq
and 
for any 
$$r\in \left[2,\frac{2N}{N-2}\right], \textrm{ if }N\geq 3 \ \textrm{ and } \ r\in \left[2,\infty\right) \textrm{ if }N=1,2,$$
there exists $\frac32<s<2$ such that $H^s(\R^N)\hookrightarrow W^{1,r}(\R^N)$. 
Thus we have for $j=3,4$ that $W^{1,\rho_j'}(\R^N)\hookrightarrow L^2(\R^N)$
and $H^{s_j}(\R^N)\hookrightarrow W^{1,r_j}(\R^N)$, for some $s_j<2$. Setting
$s=\max\{s_3,s_4\}$, from the inequality \eqref{g-grad} we obtain estimate \eqref{g-H2}
\end{proof}
\section{Global existence}\label{sec-global}
In order to complete the proof of Theorem~\ref{global} we need
 to prove that the solutions given by Theorem~\ref{teo-local}
are global. We do this by establishing an appropriate estimate for $\LL{w(t)}2$.
We distinguish three subcases, associated to the different assumptions on $W$.

\begin{proof}[Proof of Theorem~\ref{global}-$(i)$-$(a)$]
We recall that by Theorem~\ref{teo-local}
we already have the conservation of energy
\bq\label{dem-global}
E_0=\frac12 \int_{\R^N}\abs{\grad w(t)+\grad \phi}^2\,dx  +\frac14 \int_{\R^N}(W*(\abs{\phi+w(t)}^2-1))(\abs{\phi+w(t)}^2-1)\,dx,
\eq
for any $t\in (-T_{\min},T_{\max})$. Since we are assuming that $W$
is a positive definite distribution, the potential energy, i.e. the second integral in
\eqref{dem-global}, is nonnegative. Hence
$$\frac12 \int_{\R^N}\abs{\grad w(t)+\grad \phi }^2 \,dx \leq E_0$$
and using the elementary inequality  \bq\label{ab}\int_{\R^N}\abs{\grad w\grad \phi}\,dx\leq \frac14 \LL{\grad w}{2}^2+ \LL{\grad\phi}{2}^2,\eq
we conclude that
\bq \label{grad}
\LL{\grad w(t)}{2}^2\leq 4E_0+2 \LL{\grad \phi}{2}^2,  \quad t\in (T_{\min},T_{\max}),
\eq
which gives a uniform bound for $\LL{\grad w(t)}{2}$.
Therefore we only need an appropriate bound  for $\LL{w(t)}{2}$  to conclude that
\begin{equation}\label{sup}
\sup\{\HH{w(t)} : t\in (-T_{\min},T_{\max}) \}<\infty.
\end{equation}
In virtue of the blow-up alternative in Theorem~\ref{teo-local},
we will deduce from \eqref{sup} that  $T_{\max}=T_{\min}=\infty$, which
will complete the proof.

Now we prove the bound for $\LL{w(t)}{2}$.
For any $t\in(-T_{\min},T_{\max})$, we
multiply (in the $H^{-1}-H^1$ duality sense) the
 equation \eqref{GP1} by $i w$, to get
\begin{align*}
\frac12 \frac{d}{dt}\LL{w(t)}{2}^2=&\Re \int_{\R^N}if(w(t))\overline w(t)\,dx \\
=&-\Im \int_{\R^N}(\Delta\phi+ \phi  (W*(1-\abs{\phi+w(t)}^2)) \overline w(t)\,dx.
\end{align*}
Then
\begin{ecu0}\label{der-w}
\frac12 {\left|\frac{d}{dt}\LL{w(t)}{2}^2\right|} \leq& \LL{\Delta \phi}{2}\LL{w(t)}{2}+
\LL{\phi}{\infty}\int_{\R^N} \abs{W*(\abs{\phi+w(t)}^2-1)} \abs{w(t)}\,dx.
\end{ecu0}
We bound the last integral in \eqref{der-w} by $H_1(t)+H_2(t)$, with
\begin{align*}
H_1(t)&=\int_{\R^N} \abs{W*(\abs{\phi}^2-1+2  \llave{\phi, w(t)}   )} \abs{w(t)}\,dx ,\\
H_2(t)&=\int_{\R^N} \abs{W*\abs{w(t)}^2}\abs{w(t)}\,dx.
\end{align*}
Since $W\in \M_{2,2}(\R^N)$,
\begin{align*}
\abs{H_1(t)}\leq&\LL{W*(\abs{\phi}^2-1+2\llave{\phi, w})}{2}\LL{w(t)}{2}\\
\leq& \norm{W}_{2,2}
\big(\LL{\abs{\phi}^2-1}{2}+2\LL{\phi}{\infty}\LL{w(t)}{2} \big) \LL{w(t)}{2}.
\end{align*}
Therefore we have
\bq \label{cota-H1}
\abs{H_1(t)}\leq C(W,\phi)(1+\LL{w(t)}{2}^2).
\eq
If $N\geq 4$, by Lemma~\ref{last} and the Sobolev embedding theorem,
\begin{align*}
\abs{H_2(t)}\leq&\LL{W*\abs{w(t)}^2}{2} \LL{w(t)}{2}\\
\leq& C(W)  \LL{w(t)}{\frac{2N}{N-2}}^2 \LL{w(t)}{2}\\
\leq& C(W)  \LL{\grad w(t)}{2}^2 \LL{w(t)}{2}.
\end{align*}
By \eqref{grad} we conclude that
\bq \label{cota-H2}
\abs{H_2(t)}\leq C(W,\phi,E_0) \LL{w(t)}{2}, \quad  \textrm{ for all } N\geq 4. \eq

If $N=2,3$, we only  need  to use that $W\in \M_{2,2}(\R^N)$, together
with the Gagliardo-Nirenberg inequality. In fact,
\begin{align*}
\abs{H_2(t)}\leq&\LL{W*\abs{w(t)}^2}{2} \LL{w(t)}{2}\\
\leq& C(W)  \LL{w(t)}{4}^2\LL{w(t)}{2}\\
\leq& C(W) \LL{\grad w(t)}{2}^{\frac{N}2}\LL{w(t)}{2}^{3-\frac{N}{2}}.
\end{align*}
Since we are considering $N=2,3$, using \eqref{grad} it follows that
\bq \label{cota-H2-2}
\LL{H_2(t)}{2}\leq C(W,\phi,E_0)(1+ \LL{w(t)}{2}^2), \qquad N=2,3. \eq
 From inequalities \eqref{der-w}--\eqref{cota-H2-2} we have that for any $N\geq 2$
\begin{equation}
{\left|\frac{d}{dt}\LL{w(t)}{2}^2\right|}\leq C(W,\phi,E_0)(1+\LL{w(t)}{2}^2), \quad t\in(-T_{\min},T_{\max}).
\label{dem-der-w}
\end{equation}
By Gronwall's lemma we conclude that
\bqq \LL{w(t)}{2}\leq C(W,\phi,E_0) e^{C(W,\phi,E_0)\abs{t}}(1+\LL{w_0}{2}) ,t\in(-T_{\min},T_{\max}).\eqq
As we discussed before, this estimate implies \eqref{sup}, which finishes
the proof if $W$ is positive definite.\qed\newline

\begin{remark}
{We note that the argument given in the proof Theorem~\ref{global}-(i)-(a)} fails
 in dimension $N=1$. In this case if we apply the Gagliardo-Nirenberg 
inequality to $H_2$, instead of \eqref{dem-der-w} we obtain a bound for 
$\LL{w(t)}{2}^2$ in terms of $\LL{w(t)}{2}^{5/2}$, which prevents
to conclude applying Gronwall's lemma.
\end{remark}

\noindent{\em Proof of Theorem~\ref{global}-$(i)$-$(b).$}
In the case that $W$ is a positive distribution, we
cannot infer from \eqref{dem-global} a uniform bound on $\LL{\grad w(t)}{2}$.
However, using that $W\in \M_{1,1}(\R^N)$,  we will see that $\LL{\grad w(t)}{2}$ can be bounded in terms
of $\LL{w(t)}{2}$
and that we may deduce an inequality such as \eqref{dem-der-w} (without assuming that $\LL{\grad w(t)}{2}$ is a priori bounded). Then the conclusion
follows as before.

Let $A=4\LL{\phi}{\infty}+1$. Setting
\begin{align*}
w_A(x,t)=w(x,t)\chi(  \{y\in\R^N : \abs{w(y,t)}\leq A \} )(x),\\	
w_{A^c}(x,t)=w(x,t)\chi(  \{y\in\R^N : \abs{w(y,t)}> A \})(x) ,
\end{align*}
where $\chi$ is the characteristic function, we deduce that
$w=w_A+w_{A^c}$, $\abs{w}=\abs{w_A}+\abs{w_{A^c}},$	$\abs{w}^2=\abs{w_A}^2+\abs{w_{A^c}}^2$
and
\begin{equation}
\int_{\R^N}(W*(\abs{\phi+w(t)}^2-1))({\abs{\phi+w(t)}^2-1})\,dx =I_1(t)+I_2(t)+I_3(t),
\label{sum-I}
\end{equation}
with
\begin{align*}
I_1(t)=&\int_{\R^N}(W*(\abs{\phi}^2-1+2\llave{\phi, w(t)}))(\abs{\phi}^2-1+2\llave{\phi, w(t)})\,dx \\
&+2\int_{\R^N}(W*\abs{w(t)}^2)(\abs{\phi}^2-1)\,dx ,\\
I_2(t)=&\int_{\R^N}(W*\abs{w(t)}^2)(4\llave{\phi, w_A(t)}+ \abs{w_A(t)}^2) \,dx,\\
I_3(t)=&\int_{\R^N}(W*\abs{w(t)}^2)(4\llave{\phi, w_{A^c}(t)}+\abs{w_{A^c}(t)}^2)\,dx.
\end{align*}
Notice that we have used that $W$ is even to decompose it in terms of $I_1$, $I_2$ and $I_3$.
Since the energy \eqref{dem-global} is conserved in the maximal interval $(-T_{\min},T_{\max})$,
using \eqref{ab} and \eqref{sum-I}, we have that for any $t\in (-T_{\min},T_{\max})$,
\begin{equation}
\LL{\grad w(t)}{2}^2+I_3(t)\leq \abs{I_1(t)}+\abs{I_2(t)}+4\abs{E_0}+2\LL{\grad\phi}{2}^2.
\label{cota-grad}
\end{equation}
Since $W$ is a positive distribution, the choice of $A$ implies that
\begin{ecu0}	\label{I3-J3}
I_3(t)&\geq \int_{\R^N}(W*\abs{w(t)}^2)\abs{w_{A^c}(t)}(\abs{w_{A^c}(t)}-4\LL{\phi}{\infty})\,dx \\
& \geq \int_{\R^N} (W*\abs{w(t)}^2)\abs{w_{A^c}(t)}\,dx  \geq 0,
\end{ecu0}
so that $I_3$ is nonnegative. Using that $W\in \M_{1,1}(\R^N)$ we also have
\begin{ecu0}
\abs{I_1(t)}\leq& 	\norm{W}_{2,2}(\LL{\abs{\phi}^2-1}{2}+2\LL{\phi}{\infty}\LL{w}{2})^2 +2\norm{W}_{1,1} \LL{w}{2}^2(\LL{\phi}{\infty}^2+1) \label{cota-I1}
\end{ecu0}
and
\begin{align} \label{cota-I2}
\abs{I_2(t)}\leq \norm{W}_{1,1} (4A\LL{\phi}{\infty}+A^2)\LL{w(t)}{2}^{2}.
\end{align}
From inequalities \eqref{cota-grad}, \eqref{cota-I1} and \eqref{cota-I2}, we obtain that
\begin{equation}
\LL{\grad w(t)}{2}^2+I_3(t)\leq C({W},\phi,E_0)(1+\LL{w(t)}{2}^2),
\label{cota-apriori}
\end{equation}
for any $t\in (-T_{\min},T_{\max})$.

Let us set
\begin{align*}
J_1(t)&=\int_{\R^N} \abs{(W*(\abs{\phi}^2-1+2\llave{\phi, w(t)})) w(t)}\,dx ,\\
J_2(t)&=\int_{\R^N} \abs{(W*\abs{w(t)}^2)w_A(t)}\,dx ,\\
J_3(t)&=\int_{\R^N} \abs{(W*\abs{w(t)}^2)w_{A^c}(t)}\,dx.
\end{align*}
Then the last integral in \eqref{der-w} is bounded by $J_1(t)+J_2(t)+J_3(t)$.
As before, we conclude that
\begin{equation}
J_1(t)+J_2(t)\leq C(W,\phi)(1+\LL{w(t)}2^2).
\label{J1+J2}
\end{equation}
From \eqref{I3-J3} we have $J_3(t)\leq I_3(t)$. Then \eqref{cota-apriori} and  \eqref{I3-J3} imply that
\begin{equation}
J_3(t) \leq C(W,\phi,E_0)(1+\LL{w(t)}{2}^2).
\label{J3}
\end{equation}
The estimates \eqref{J1+J2} and \eqref{J3}, together with \eqref{der-w}, provide again
the inequality \eqref{dem-der-w}, and then the proof
is completed as in the previous case.\qed\newline

\noindent{\em Proof of Theorem~\ref{global}-$(ii).$}
As before, the local well-posedness follows from Theorem~\ref{teo-local}.
{Moreover, from Theorem~\ref{global}-(i)-(a) we have the global well-posedness for} $N\geq 2$.
From  Proposition~\ref{lema-def} we have that $W$ is a positive definite distribution
and, as shown before, this implies that $\norm{\grad{w}(t)}_{L^2}$
is uniformly bounded in the maximal interval ${(-T_{\min},T_{\max})}$ in terms
of $E_0$ and $\phi$ (see inequality \eqref{grad}). Then it only remains to prove the inequality \eqref{linear},
for $t\in {(-T_{\min},T_{\max})}$.

The argument follows the lines of the proof in \cite{banica} for the local Gross-Pitaevskii
 equation. For sake of completeness we give the details.

Since  $W$ is positive definite, from the conservation of energy we have
\bq\label{L3}
0\leq \int_{\R^N}(W*(\abs{\phi+w(t)}^2-1))(\abs{\phi+w(t)}^2-1)\,dx \leq 4E_0.
\eq
On the other hand, Lemma~\ref{lema-L2} gives a lower bound for the potential energy
\bq\label{L2}
 \sigma \LL{\abs{\phi+w(t)}^2-1}{2}^2\leq \int_{\R^N}(W*(\abs{\phi+w(t)}^2-1))(\abs{\phi+w(t)}^2-1)\,dx. \eq
From \eqref{der-w} and using H\"older inequality we obtain
\begin{ecu0}
\frac12 \left|{\frac{d}{dt}\LL{w(t)}{2}^2}\right|\leq& \LL{\Delta \phi}{2}\LL{w(t)}{2}+
\norm{W}_{2,2}\LL{\phi}{\infty} \LL{\abs{\phi+w(t)}^2-1}{2} \LL{w(t)}{2}. \label{fin}
\end{ecu0}
Thus from \eqref{L3}, \eqref{L2} and \eqref{fin}, we have that for any $\delta>0$
$$\frac12 \left|{\frac{d}{dt}(\LL{w(t)}{2}^2+\delta )}\right|\leq (\LL{w(t)}{2}^2+\delta)^\frac12
 \left(\LL{\Delta \phi}{2}+\norm{W}_{2,2}\LL{\phi}{\infty}\sqrt{\frac{4E_0}{\sigma}}\right).$$
 Dividing by $\LL{w(t)}{2}^2+\delta>0$, {integrating and then taking }$\delta\to 0$
we conclude that
\bq\label{fin2} \LL{w(t)}{2}\leq \left(\LL{\Delta \phi}{2}+\norm{W}_{2,2}\LL{\phi}{\infty}\sqrt{\frac{4E_0}{\sigma}}\right) \abs{t}+\LL{w_0}{2},\eq
for any $t\in {(-T_{\min},T_{\max})}$.
As discussed before, this implies that $\norm{w(t)}_{H^1}$ is uniformly bounded in
$(-T_{\min},T_{\max})$.  Therefore by the blow-up alternative, we infer that $T_{\min}=T_{\max}=\infty$.
Since $u(t)=w(t)+\phi$ and $u_0=w_0+\phi$, \eqref{fin2} implies
\eqref{linear}, finishing the proof.
\end{proof}
\section{Equation (\ref{NGP}) in energy space}\label{gallo}
We recall the following results about the energy space $\E(\R^N)$.
We refer to \cite{gerard,gerard3,gallo} for their proofs.
\begin{lema}\label{lema1}
Let $u\in \E(\R^N)$. Then there exists $\phi\in C_b^\infty\cap \E(\R^N)$
with $\grad \phi\in H^\infty(\R^N)$, and $w\in H^1(\R^N)$ such that $u=\phi+w$.
\end{lema}

\begin{lema}\label{lema2}
Let $1\leq N\leq 4$. Then $\E(\R^N)$ is a complete metric space
with the distance \eqref{distance}, $\E(\R^N)+H^1(\R^N) \subset \E(\R^N)$
and the maps
\begin{gather*}
u\in \E(\R^N) \mapsto \grad u\in L^2(\R^N),\ u\in \E(\R^N) \mapsto 1-\abs{u}^2\in L^2(\R^N), \\
 (u,w)\in \E(\R^N)\times H^1(\R^N) \mapsto u+w\in \E(\R^N) 
\end{gather*}
are continuous.
\end{lema}

\begin{lema}\label{lema4}
{Assume} $1\leq N\leq 4$. Let $W\in\M_{2,2}(\R^N)$, $u\in C(\R,\E(\R^N))$, $v\in C(\R,L^2(\R^N))$ and
$$\Phi(t):= \int_{0}^t e^{i(t-s)\Delta}u(s)(W*v(s))\,ds, \quad t\in [0,T].$$
Then $\Phi\in C([0,T],L^2(\R^2))$
and there exists a  universal constant $C$ such that 
$$\norm{\Phi}_{L^\infty((0,T),L^2)}\leq C\max\{T,T^{\frac{8-N}{N}}\} \norm{W}_{2,2}
(\norm{1-\abs{u}^2}_{L^\infty((0,T),L^2)}+\norm{\grad u}_{L^\infty((0,T),L^2)})
\norm{v}_{L^\infty((0,T),L^2)}.
$$
\end{lema}
\begin{proof}
By  Lemma~1 in \cite{gerard} and Lemma \ref{lema2}, we may decompose $u(t)=u_1(t)+u_2(t),$
with $\norm{u_1}_{L^\infty(\R,L^\infty)}\leq 3$ and 
\bq \label{last1} 
\norm{u_2}_{L^\infty((0,T),H^1)}\leq C(\norm{1-\abs{u}^2}_{L^\infty((0,T),L^2)}+\norm{\grad u}_{L^\infty((0,T),L^2)}).
\eq
Let us set 
$$\Phi_j(t):= \int_{0}^t e^{i(t-s)\Delta}u_j(s)(W*v(s))\,ds, \quad j=1,2.$$
By the Strichartz estimates we have that $\Phi_1\in C([0,T],L^2(\R^2))$ and
\bq \label{last2}
\norm{\Phi_1}_{L^\infty((0,T),L^2)}\leq C T \norm{W}_{2,2} \norm{v}_{L^\infty(\R,L^2)}.\eq
Since $(8/N,4)$ is an admissible Strichartz pair in dimension $1\leq N\leq 4$, we also 
infer that $\Phi_2\in C([0,T],L^2(\R^2))$ and
\begin{ecu0} \label{last3}
\norm{\Phi_2}_{L^\infty((0,T),L^2)}&\leq  C T^{\frac{8-N}{N}} \norm{u(W*v)}_{L^\infty(\R,L^{4/3})} \\
&\leq C T^{\frac{8-N}{N}}\norm{W}_{2,2} \norm{u}_{L^\infty(\R,L^4)}\norm{v}_{L^\infty(\R,L^2)}
\end{ecu0}
Combining \eqref{last1}-\eqref{last3} and using 
the Sobolev embedding $H^1(\R^N)\hookrightarrow L^4(\R^N)$, the conclusion follows. 
\end{proof}

\begin{proof}[Proof of Theorem~\ref{teo-gallo}]
After Theorem~\ref{global}, the proof follows the same arguments
given in \cite{gallo}. For sake of completeness we sketch
the proof.

Given $u_0\in {\E(\R^N)}$,  by Lemma~\ref{lema1} we have that $u_0=\phi+\tilde w_0$,
for some $\tilde w_0\in H^1(\R^N)$ and $\phi$  satisfying \eqref{phi}.
Thus Theorem~\ref{global} gives a solution of \eqref{NGP} of the
form $u=\phi+\tilde w$, with $\tilde w\in C(\R,H^1(\R^N))$. Therefore $u=u_0+ w$,
with $w=\tilde w-\tilde w_0$ is the desired solution.
To prove the uniqueness in the energy space, we consider $1\leq N\leq 4$.
Let $v\in C(\R,\E(\R^N))$ be a {mild} solution of \eqref{NGP} with $v(0)=u_0$.
It is sufficient to show that $v-u_0\in C(\R,H^1(\R^N))$, because
then we may apply the uniqueness result given by Theorem~\ref{global}.
We do this by proving that $u-v\in C(\R,H^1(\R^N))$.
Note that by Lemma~\ref{lema2},
 $u\in u_0+C(\R,H^1(\R^N)) \subset C(\R,\E(\R^N))$
  and  $\grad u,\grad v\in C(\R,L^2(\R^N))$.
It only remains to prove that $u-v\in C(\R,L^2(\R^N))$. Let $T>0$
and $t\in [0,T]$, then
\bqq
u(t)-v(t)=i \int_{0}^t e^{i(t-s)\Delta}(G(u(s))-G(v(s)))\,ds,
\eqq
with
\bqq
G(u)-G(v)=u (W*(\abs{v}^2-\abs{u}^2))+(u-v)(W*(1-\abs{v}^2)).
\eqq
{Applying Lemma \ref{lema4} to} $u(W*(\abs{v}^2-\abs{u}^2))$ and 
$(u-v)(W*(1-\abs{v}^2))$, we conclude that $u-v\in C([0,T],L^2(\R^N))$.
\end{proof}

\section{Other conservation laws}\label{conservation}
In this section we consider a global solution $u$ of \eqref{NGP}
given by Theorem~\ref{global}. We have already seen that the energy is conserved
by the flow of this solution. Now we discuss the notions
of momentum and mass associated to the equation \eqref{NGP},
that are also formally conserved.

\subsection{The momentum}
The vectorial momentum for \eqref{NGP} is given by
\bq\label{def-mom}     p(u)=\frac{1}{2}\int_{\R^N}{\llave{i\grad u ,{u}}}\,dx.\eq
A formal computation shows that the derivative of the momentum is zero and
thus it is a conserved quantity. Moreover, if $u=\phi+w$ we have
\begin{align*}
    p (u)=&\frac{1}{2}\int_{\R^N}{\llave{i\grad \phi ,\phi}}\,dx+
\frac{1}{2}\int_{\R^N}\llave{i\grad w ,w}\,dx\\
&+ \frac{1}{2}\int_{\R^N}{\llave{i\grad \phi ,w}}\,dx+
\frac{1}{2}\int_{\R^N}\llave{i\grad w ,\phi}\,dx.
\end{align*}
Here the problem is that ${\llave{i\grad \phi ,\phi-1}}$ and
${\llave{i\grad w ,\phi-1}}$  are not necessarily integrable
for  $w\in C(\R,H^1(\R^N))$.
However, a formal integration by parts yields
\bq\label{mom-for}
    p (u)=\frac{1}{2}\int_{\R^N}{\llave{i\grad \phi ,\phi}}\,dx+
\frac{1}{2}\int_{\R^N}{\llave{i\grad w ,w}}\,dx+
\int_{\R^N}{\llave{i\grad \phi ,w}}\,dx,
\eq
reducing the ill-defined term to ${\llave{i\grad \phi ,\phi}}$,
supposing that we can justify the integration by parts.
In order to give a rigorous sense to these computations,
we use the following definition proposed by Mari\c{s} in \cite{maris}.
\begin{definition}
Let $\X(\R^N)=\{\grad v : v\in \H^1(\R^N) \}$
and $\X_j(\R^N)=\{\partial_j v : v\in\H^1(\R^N) \}$, with $j=1,\dots, N.$
For any $h_1\in L^1(\R^N)$ and $h_2\in \X_j(\R^N)$
we define the linear operator $L_j$ on $L^1(\R^N)+\X_j(\R^N)$ by
$$L_j(h_1+h_2)=\frac{1}{2}\int_{\R^N}h_1~dx.$$
\end{definition}

\begin{lema}\label{lema-int}
Let $N\geq 2$ and $j\in\{1,\dots,N\}$. Then
  $$\int_{\R^N}h=0, \quad \textup{for any }h\in L^1(\R^N)\cap \X_j(\R^N).$$
In particular $L_j$ is a well-defined linear continuous operator on $L^1(\R^N)+\X_j(\R^N)$
in any dimension $N\geq 2$.
\end{lema}
\begin{proof}
The proof of Lemma~\ref{lema-int} is given by Mari\c{s} (Lemma 2.3 in \cite{maris}) in the
case $N\geq 3$. The same argument works in dimension two, provided that
a function in $\H^1(\R^2)$ defines a tempered distribution. In fact,
this last point was shown by G\'erard  (see \cite{gerard3}, p.~8), concluding the proof.
\end{proof}

Following the ideas proposed in \cite{maris} in dimension $N\geq 3$,
we have the following result that is essential to define our notion of momentum.
\begin{lema}\label{des-phi}
Let $N\geq 2$, $j=1,\dots, N$ and $w\in H^1(\R^N)$. Then ${\llave{i \partial_j \phi,\phi}} \in L^1(\R^N)+\X_j(\R^N)$,
\mbox{${\llave{i \partial_j \phi,w}} \in L^1(\R^N)$},  ${\llave{i\phi,\partial_j w}} \in L^1(\R^N)+\X_j(\R^N)$ and
\bq\label{ipp} L_j(\llave{i\partial_j \phi,w})=-L_j({\llave{i \phi,\partial_j w}} ).\eq
\end{lema}
\begin{proof}
The assumption \eqref{phi} implies that there is
a radius $R>1$ such that
$\abs{\phi(x)}\geq \frac{1}{2}$, for all $x\in B(0,R)^c$
and  $\phi$ is $C^1$ in $B(0,R)^c$.
Then,   there are some scalar functions $\tilde \rho,\tilde \theta\in C^1(B(0,R)^c)\cap H^1_{\textup{loc}}(B(0,R)^c)$
such that
$$\phi=\tilde \rho e^{i\tilde \theta},\qquad \textup{ on } B(0,R)^c.$$
Moreover, since $\partial_j\phi\in L^2(\R^N)$ and
$$\abs{\partial_j\phi}^2=\abs{\partial_j \tilde \rho}^2+{\tilde \rho}^2\abs{\partial_j \tilde{\theta}}^2,\quad \textup{ on } B(0,R)^c$$
we deduce that $\partial_j \tilde \rho,\partial_j \tilde \theta \in L^2(B(0,R)^c)$.
By Whitney extension theorem (cf.~\cite{krantz}, p.~167), there exist scalar functions $\rho$, $\theta \in C^1(\R^N)$
such that $\rho=\tilde \rho$ and $\theta=\tilde  \theta $ on $B(0,R)^c$. Setting
$$\phi_1=\rho e^{i \theta}\quad \textup{ and }\quad \phi_2=\phi-\phi_1,$$
we have
\bq\label{des-mom} {\llave{i \partial_j \phi,\phi}}={\llave{i \partial_j \phi_1,\phi_1}}+
{\llave{i \partial_j \phi_1,\phi_2}}+
{\llave{i \partial_j \phi_2,\phi_1}}+
{\llave{i \partial_j \phi_2,\phi_2}}.\eq
Since  $\supp\phi_2,\supp\grad \phi_2 \subset \bar B(0,R)$, the last three terms in the r.h.s. of \eqref{des-mom} belong to $L^1(\R^N)$.
For the remaining term, a direct computation gives
\bq\label{des1}
{\llave{i \partial_j \phi_1,\phi_1}}=- \rho^2 \partial_j \theta=(1-\rho^2)\partial_j\theta-\partial_j\theta, \quad \textup{ on }\R^N.\eq
The fact that $\partial_j \tilde \theta \in L^2(B(0,R)^c)$
implies that $\partial_j\theta \in L^2(\R^N)$ and
from \eqref{phi} it follows that ${\abs{\rho}^2-1}\in L^2(\R^N)$. Therefore from \eqref{des1}
we conclude that ${\llave{i \partial_j \phi_1,\phi_1}} \in L^1(\R^N)+\X_j(\R^N)$ and hence
${\llave{i \partial_j \phi,\phi}} \in L^1(\R^N)+\X_j(\R^N)$.

To finish the proof, we notice that from \eqref{phi} and the above 
computations we also have that
 $\phi_1\in \X(\R^N)\cap C^1(\R^N)\cap W^{1,\infty}(\R^N)$ and $\phi_2\in H^1(\R^N)$. Then
a slight modification of the argument given in Lemma 2.5 in \cite{maris}, allows us
to deduce that ${\llave{i \partial_j \phi,w}} \in L^1(\R^N)$,  ${\llave{i\phi,\partial_j w}} \in L^1(\R^N)+\X_j(\R^N)$ and
the identity \eqref{ipp}.
\end{proof}

In virtue of Lemma \ref{des-phi} and making an analogy to \eqref{def-mom}, for $N\geq 2$ and $u\in \phi+H^1(\R^N)$, we define the {\em generalized momentum}
$q=(q_1,\dots,q_N)$ as
\bqq q_j(u)=L_j({\llave{i\partial_j u,u}}), \quad j=1\dots, N. \eqq
Furthermore, by \eqref{ipp} we have
\bq\label{mom-for2}
 q_j(u)=L_j({\llave{i\partial_j \phi ,\phi}})+
\frac{1}{2}\int_{\R^N}{\llave{i\partial_j w ,w}}\,dx+
\int_{\R^N}{\llave{i\partial_j\phi ,w}}\,dx,
\eq
which can be seen as a rigorous formulation of \eqref{mom-for}.

In dimension one, the operator $L_j$ is not well-defined. In fact, following
the idea of the proof of Lemma \ref{des-phi}, if we assume that
$u=\rho e^{i\theta}$ then
\bqq{\llave{i u',u}}=- \rho^2 \theta'=(1-\rho^2)\theta'-\theta'.\eqq
Supposing that $\lim\limits_{R\to\infty}(\theta(R)-\theta(-R))$ exists, we would have
\bq\label{theta-2} \int_{\R}\theta'(x)\,dx =\lim_{R\to\infty}(\theta(R)-\theta(-R)).\eq
Thus we necessarily need to modify the definition of the momentum in the one-dimensional
case to take into account the phase change \eqref{theta-2}. This approach is taken in
\cite{bethuel-black} using the following notion of untwisted momentum.
\begin{definition} For $u\in \phi+H^1(\R)$, we define the
operator $\mathcal L$ on $\phi+H^1(\R)$ by
\bq\label{mom1} \mathcal  L(u)=\lim\limits_{R\to\infty}\left( \frac12\int_{-R}^{R} {\llave{i u' ,u}}\,dx+\frac12\left(\arg u(R)-\arg u(-R)\right) \right)\mod \pi\eq
\end{definition}
In \cite{bethuel-black} it is proved that the limit in \eqref{mom1} actually exists. 
Therefore,
as in the higher dimensional case, we define the {\em generalized momentum
in dimension one} as
$$q_1(u)=\mathcal  L(u).$$
The following result shows that this definition give{s} us
an analogous expression to \eqref{mom-for2}.
\begin{lema}[\cite{bethuel-black}]\label{mom-dim-1} Let $u=\phi+w$, $w \in H^1(\R)$. Then
\bqq  q_1(u)=\mathcal  L(\phi)+\frac{1}{2}\int_{\R}{\llave{i w' ,w}}\,dx+\int_{\R}{\llave{i \phi' ,w}}\,dx.\eqq
\end{lema}
Now that we have explained the notion of generalized momentum  in any
dimension, we can proceed to prove Theorem~\ref{teo-momentum}.

\begin{proof}[Proof of Theorem \ref{teo-momentum}]
In view of the continuous dependence of the flow, Lemma~\ref{mom-dim-1}, \eqref{mom-for2} and Proposition~\ref{prop-H2},
 we only need to prove the conservation of momentum for
$u_0=\phi+w_0$, with $w_0\in H^2(\R^N)$. Thus we assume that $u-\phi=w\in C(\R,H^2(\R^N))\cap C^1(\R,L^2(\R^N))$.
Integrating by parts we have that for any
$j=1,\dots,N$ and  $t\in \R$,
\begin{align*}
\partial_t q_j(u(t))
&=\partial_t \left( \frac{1}{2}\int_{\R^N}{\llave{i\partial_j w(t) ,w(t)}}\,dx+\int_{\R^N}{\llave{i\partial_j \phi ,w(t)}}\,dx \right)\\
&=\int_{\R^N}{\llave{i\partial_j (w(t)+\phi) ,\partial_tw(t)}}\,dx\\
&=\int_{\R^N}{\llave{i\partial_j u(t) ,\partial_t u(t)}}\,dx\\
&=\int_{\R^N}{\llave{\partial_j u(t) , \Delta u(t)+u(t)(W*(1-\abs{u(t)}^2))}}\,dx.
\end{align*}
Since $\abs{\grad u(t)}^2\in W^{1,1}(\R^N)$, an integration by parts leads to
\begin{ecu0}\label{der-mom2}
\partial_t q_j(u(t))
&={-}\frac{1}{2}\int_{\R^N} \partial_j \abs{\grad u(t)}^2\,dx+ \int_{\R^N}(W*(1-\abs{u(t)}^2)){\llave{u(t),\partial_j u(t)}}\,dx\\
&=\int_{\R^N}(W*(1-\abs{u(t)}^2)){\llave{u(t),\partial_j u(t)}}\,dx.
\end{ecu0}
Now we notice that
\bq\label{dar-W} \partial_j\!\left((1-\abs{u}^2) (W*(1-\abs{u}^2)) \right)=-2{\llave{u,{\partial_j u}}}(W*(1-\abs{u}^2))-2(1-\abs{u}^2)(W*{\llave{u,{\partial_j u}}}).
\eq
From \eqref{dar-W} and Lemma~\ref{lema-int-finita}, we have
\bq\label{dar-W2}\int_{\R^N}{\llave{u,{\partial_j u}}}(W*(1-\abs{u}^2))\,dx=\int_{\R^N} (1-\abs{u}^2)(W*{\llave{u,{\partial_j u}}})\,dx.
\eq
Since $\left( (1-\abs{u(t)}^2) (W*(1-\abs{u(t)}^2)) \right)\in W^{1,1}(\R^N)$, from \eqref{der-mom2}, \eqref{dar-W} and \eqref{dar-W2}
we infer that
$$\partial_t q_j(u(t))={-}\frac14 \int_{\R^N} \partial_j\left((W*(1-\abs{u(t)}^2))(1-\abs{u(t)}^2)\right)\,dx=0,$$
concluding the proof.
\end{proof}
\begin{remark}
This argument also proves the conservation of momentum
stated in Theorem \ref{teo-local}.
\end{remark}

\subsection{The mass}
In a recent article, B\'ethuel et al. \cite{bethuel3} give a  definition
for the mass for the local Gross-Pitaevskii equation in the one-dimensional
case. In this subsection we try to extend this notion to higher dimensions.

Let $\chi\in C^\infty_0(\R;\R)$ be a function such that
$\chi(x)=1$ if $\abs{x}\leq 1$,
$\chi(x)=0$ if $\abs{x}\geq 2$ and $\LL{\chi'}{\infty},\LL{\chi''}{\infty}\leq 2.$
For any $R>0$, $a\in \R^N$, we set
$$\chi_{a,R}(x)=\chi\left(   \frac{\abs{x-a}}{R}\right),\quad  x\in \R^N$$
and the quantities
\bqq m^+(u)=\inf_{a\in \R^N}\limsup_{R\to \infty}\int_{\R^N}(1-\abs{u}^2)\chi_{a,R}\,dx, \quad
m^-(u)=\sup_{a\in \R^N}\liminf_{R\to \infty}\int_{\R^N}(1-\abs{u}^2)\chi_{a,R}\,dx.
\eqq
In the case that $1-\abs{u}^2 \in L^1(\R^N)$, $m^{+}(u)=m^{-}(u)$. More generally, if $u$ is such that
$m^{+}(u)=m^{-}(u)$, we define the {\em generalized mass} as
\bqq  m(u)\equiv m^{+}(u)=m^{-}(u).\eqq

The following result is a more accurate version of Theorem~\ref{teo-masa} and shows that the generalized mass is conserved
if $N\leq 4$. However, we need a faster decay for $\phi$ in dimensions
three and four, which is at least satisfied by the travelling waves
in the local problem (see \cite{gravejat-decay}).

\begin{teo}
Let $1\leq N\leq 4$. In addition to \eqref{phi}, assume that
$\grad \phi\in L^\frac{N}{N-1}(\R^N)$ if $N=3,4$.
Suppose that $u_0\in \phi+H^1(\R^N)$ with $m^+(u_0)$ \lp respectively $m^{-}(u_0)$\rp  finite. Then
the associated solution of \eqref{NGP} given by Theorem~\ref{global} satisfies
$m^+(u(t))=m^+(u_0)$ \lp respectively $m^-(u(t))=m^-(u_0)$\textup{)}, for any $t\in \R$. In particular,
if  $u_0$ has finite generalized mass, then the generalized mass is conserved by the flow, that is
$ m(u(t))=m(u_0)$, for any $t\in\R$.
\end{teo}

\begin{proof}
Let $u_0=\phi+w_0$ and $u=\phi+w$, $w_0\in H^1(\R^N)$, $w\in C(\R,H^1(\R^N))\cap C^1(\R,H^{-1}(\R^N))$. We take a sequence
 $w_{0,n}\in H^2(\R^N)$ such that $w_{0,n}\to w_0$ in $H^1(\R^N)$.
By Proposition~\ref{prop-H2} and the continuous dependence property of
Theorem~\ref{global}, the solutions
$u_n=\phi+w_n$ of \eqref{NGP} with initial data $\phi+w_{0,n}$ satisfy
\bq\label{conver-masa} w_n\in C(\R,H^2(\R^N))\cap C^1(\R,L^2(\R^N)) \quad \textup{ and }\quad w_n\to w \textup{ in } C(I,H^1(\R^N)),\eq
for any bounded closed interval $I$.

Setting $\eta(t)=1-\abs{u(t)}^2$, $\eta_n(t)=1-\abs{u_n(t)}^2$
and using that the functions $u_n$ are solution of \eqref{NGP}, it follows
\bqq \partial_t \eta_n(t)=-2\Re(i\overline{u}_n(t)\Delta u_n(t)).\eqq
Then integrating by parts
\bq\label{der-eta}
\partial_t \left( \int_{\R^N}\eta_n(t) \chi_{a,R} \, dx\right)=\int_{\R^N}\partial_t \eta_n(t) \chi_{a,R} \,dx
=I_1(t)+I_2(t)+I_3,\eq
with
\begin{align*}
I_1(t)&=-2\Im \int_{\R^N} (\overline w_n(t) \grad w_n(t)+\overline w_n(t) \grad \phi)\grad \chi_{a,R} \,dx,\\
I_2(t)&=-2\Im \int_{\R^N}\overline \phi \grad w_n(t) \grad \chi_{a,R} \,dx,\\
I_3&=-2\Im \int_{\R^N}\overline \phi \grad \phi \grad \chi_{a,R} \,dx.
\end{align*}
Noticing that $\LL{\Delta \chi_{a,R}}{2}\leq C R^{\frac{N-4}{2}},$
we have that $\LL{\grad \chi_{a,R}}{\infty}$
and $\LL{\Delta \chi_{a,R}}{2}$
are uniformly bounded in $a$ and $R$. Setting
$$\Omega_{a,R}=\{x\in \R^N : R<\abs{x-a}<2R \}$$
and using the Cauchy-Schwarz inequality we have
\bq\label{cota-1}
\abs{I_1(t)}\leq C(\phi) \norm{w_n(t)}_{L^2(\Omega_{a,R})}( \norm{\grad w_n(t)}_{L^2(\Omega_{a,R})}+1).
\eq
For $I_2$, we first integrate by parts
\bqq {I_2(t)}=2\Im \int_{\R^N}w_n(t)( \grad \overline\phi \grad \chi_{a,R}+\overline \phi \Delta \chi_{a,R}) \,dx,\eqq
thus
\bq\label{cota-2} \abs{I_2(t)}\leq C(\phi) \norm{w_n(t)}_{L^2(\Omega_{a,R})}.\eq
Using H\"older inequality, it follows that
\bq\label{cota-3} \abs{I_3}\leq
\begin{cases}
\LL{\phi}{\infty}\norm{\grad \phi}_{L^2(\Omega_{a,R})} \LL{\grad \chi_{a,R}}{2}, &\textup{ if }N=1\\
\LL{\phi}{\infty}\norm{\grad \phi}_{L^\frac{N}{N-1}(\Omega_{a,R})} \LL{\grad \chi_{a,R}}{N\vphantom{L\frac{N}{N-1}}}, &\textup{ if }2\leq N\leq 4.
\end{cases}
\eq
Note that the choice of $\chi$ implies that $\LL{\grad \chi_{a,R}}{N}$ is uniformly bounded in $a$ and $R$ in any dimension,
and so is $\LL{\grad \chi_{a,R}}{2}$ in dimension one. Then by putting
together \eqref{der-eta}-\eqref{cota-3}, we obtain
\bqq
\left|{\partial_t \left( \int_{\R^N}\eta_n(t) \chi_{a,R} dx\right)} \right|\leq
C(\phi)
(\norm{w_n(t)}_{{L^2(\Omega_{a,R})}}(1+\LL{\grad w_n(t)}{2})+\norm{\grad{\phi}}_{L^{N^*}(\Omega_{a,R})} ),
\eqq
with $N^*=2$ if $N=1$ and $N^*=\frac{N}{N-1}$ if $2\leq N\leq 4$.
Integrating this inequality between $0$ and $t$ and, by \eqref{conver-masa},  passing
to the limit we have
\begin{multline}\label{mas2}
 \left| \int_{\R^N}\eta(t)\chi_{a,R} \,dx -\int_{\R^N}\eta(0) \chi_{a,R} \,dx\right|\leq \\
C(\phi) \int_{0}^{\abs t}
\norm{w(s)}_{L^2(\Omega_{a,R})}(1+\LL{\grad w(s)}{2})\,ds +C(\phi)\abs{t}\norm{\grad{\phi}}_{L^{N^*}(\Omega_{a,R})}.
\end{multline}
From the proof of Theorem~\ref{global}, we deduce that for some constant $K$, depending only
on ${w_0}$, $E_0$, $\phi$ and $W,$
\bq\label{dom} \LL{w(t)}{2} \leq {K e^{K\abs{t}}},\quad \LL{\grad w(t)}{2} \leq {Ke^{K\abs{t}}}.\eq
Then, by Cauchy-Schwarz inequality,
\begin{align*}
 \int_{0}^{\abs t}
\norm{w(s)}_{L^2(\Omega_{a,R})}(1+\LL{\grad w(s)}{2})\,ds \leq&
Ke^{K\abs{t}}  \int_{0}^{\abs t} \norm{w(s)}_{L^2(\Omega_{a,R})} \,ds \nonumber \\
\leq& K e^{K\abs{t}}{\abs{t}^\frac12} \left(  \int_{0}^{\abs t}\int_{\Omega_{a,R}} \abs{w(s)}^2\,dx\,ds\right)^{\frac12}.
\end{align*}
This inequality together with \eqref{dom}, the dominated convergence theorem and
\eqref{mas2} imply that
\bqq\lim_{R\to \infty}\left( \int_{\R^N}(1-\abs{u(t)}^2)\chi_{a,R}\,dx  - \int_{\R^N}(1-\abs{u_0}^2)\chi_{a,R}\,dx \right)=0.\eqq
The conclusion follows from the definition of $m^+$, $m^-$ and $m$.
\end{proof}

An interesting open question is to extend the statement of Theorem~\ref{teo-masa}
to a more meaningful notion of mass such as
$${\mathfrak m}^+(u)=\inf_{a\in \R}\limsup_{R\to \infty}\int_{B(a,R)}(1-\abs{u}^2)\,dx, \quad  {\mathfrak m}^-(u)=\sup_{a\in \R}\liminf_{R\to \infty}\int_{B(a,R)}(1-\abs{u}^2)\,dx.$$
In fact, in the one-dimensional case, one can choose a test function $\chi$ such that
$$\norm{\chi_{a,R}}_{L^2(\supp(\grad \chi_{a,R}))}$$
is uniformly  bounded  in $a$ and $R$. Then one can see that Theorem~\ref{teo-masa}
remains true replacing $m$ by $\mathfrak m$, recovering a result of  B\'ethuel et al. (see Appendix in \cite{bethuel3}).
However, in higher dimensions we do not {know} if this is possible.

 \bibliographystyle{abbrv}   
\bibliography{ref}

\end{document}